\documentclass[nosumlimits,twoside,reqno]{amsart}
\usepackage{amssymb}
\usepackage{amsmath}
\usepackage[bookmarksnumbered,colorlinks,plainpages,hyperindex]{hyperref}
\usepackage{amsfonts}
\usepackage{color}
\usepackage{float}
\usepackage{graphicx}
\usepackage{xy}

\xyoption{all}
\newtheorem{theorem}{\sc Theorem}[section]
\newtheorem{proposition}[theorem]{\sc Proposition}

\newtheorem{corollary}[theorem]{\sc Corollary}
\theoremstyle{definition}
\newtheorem{definition}[theorem]{\sc Definition}

\newtheorem{example}[theorem]{\sc Example}

\theoremstyle{remark}
\newtheorem{remark}[theorem]{\sc Remark}

\newtheoremstyle{mystyle}{}{}{}{}{}{.}{ }
  {\thmname{#1}\thmnote{\scshape #3}}
\theoremstyle{mystyle}

\def\stac#1{\raise-.2cm\hbox{$\stackrel{\displaystyle\otimes}{\scriptscriptstyle{#1}}$}}
\def\Bstac{\raise-.2cm\hbox{$\stackrel{\displaystyle\otimes}{\scriptscriptstyle{{\bf
      R}^\chi}}$}}

\def\cten#1{\raise-.2cm\hbox{$\stackrel{\displaystyle\widehat{\otimes}}
{\scriptscriptstyle{#1}}$}}
\def\lb{{braid preserving }}
\def\bpm{{braided pair of monads }}
\def\BP{\otimes_{{\bf R}^\chi}}

\numberwithin{equation}{section}

\begin{document}
\title{Examples of para-cocyclic objects induced by \emph{BD}-laws}
\author{Gabriella B\"ohm}
\author{Drago\c s \c Stefan}
\address{Research Institute for Particle and Nuclear Physics, Budapest,
H-1525 Budapest 114, P.O.B.49, Hungary.}
\address{ University of Bucharest, Faculty of Mathematics and Computer Science,
14 Academiei Street, Bucharest,  Ro-010014, Romania.}
\date{}
\maketitle\vspace*{-5mm}
\begin{center} Dedicated to Freddy Van Oystaeyen on the
occasion of his 60th birthday.
\end{center}
\begin{abstract}
In a recent paper \cite{BohmStefan}, we gave a general construction of a
para-cocyclic structure on a cosimplex, associated to a so called admissible
septuple
-- consisting of two categories, three functors and two natural
transformations, subject to compatibility relations. The main examples of such
admissible septuples were induced by algebra homomorphisms. In this note we
provide more general examples coming from appropriate (`locally
braided') morphisms of monads.
\end{abstract}


\section*{Introduction}

History of cyclic homology started in the early eighties of the last
century. The seminal works of the pioneers A Connes, B Tsygan, D Quillen and
J-L Loday were motivated by looking for non-commutative generalizations of de
Rham cohomology on one hand, and Lie algebra homology of matrices on the other
hand.

In the subsequent decades cyclic homology has been extensively studied and
became an important tool in diverse areas of mathematics, such as homological
algebra, algebraic topology, Lie algebras, algebraic K-theory and so
non-commutative differential geometry. Thus by most various motivations, lots
of examples have been constructed. In order to study general features of the
examples, and also to be able to construct new ones, it was desirable to find
a unifying general description. A fundamental first step in this direction was
made by A Kaygun in \cite{Kay:UniHCyc}, who gave a construction of
para-(co)cyclic objects in symmetric monoidal categories in terms of
(co)monoids. In particular, in this way he managed to describe in a universal
form all examples arising from Hopf cyclic theory (upto cyclic duality,
cf. \cite{KR}). Motivated by a  
generalization to bialgebroids (over non-commutative rings, in which case the
underlying bimodule categories are not symmetric), in \cite{BohmStefan} we
made a further step of generalization and constructed para-(co)cyclic objects
in arbitrary categories, in terms of (co)monads. Kaygun's construction can be
recovered as a particular case when the (co)monads in question are induced by
(co)monoids.

Mean examples of Kaygun's construction are induced by algebras over a
commutative ring. By analogy, in this paper we show that appropriate monad
morphisms (which are `locally braided' in a sense to be described) induce
examples of para-cocyclic objects in \cite{BohmStefan}.

The paper is organized as follows. In Section \ref{sec:bimod} we
recall some facts about monads and {\em BD}-laws that are used in
the paper. In Section \ref{sec:construction} we introduce the
notion of a {\em locally braid preserving morphism} of monads,
generalizing a homomorphism of algebras, and we investigate their
basic properties that are needed to state and prove our main
result. In Section \ref{sec:main} we show that any such morphism
determines an `admissible septuple' in the sense of
\cite{BohmStefan}, hence can be used to construct para-cocyclic
objects. Here we also illustrate how this construction works in
the example of a morphism of algebras in a braided monoidal
category (i.e. when there is a {\em global} braiding). Particular examples
will be provided by appropriate homomorphisms of (co)module algebras of a
(co)quasitriangular Hopf algebra.

{\bf Acknowledgment.} It is our pleasure to thank the organizers
of the conference  {\em `Non-com\-mu\-ta\-tive Rings and
Geometry'}, Almer\'{i}a, Spain,
  18-22 September 2007, held in the honour of the 60th birthday of Freddy Van
  Oystaeyen. The first author was partially supported by the Hungarian
  Scientific Research Fund OTKA K 68195 and the Bolyai J\'anos Research
  Scholarship. The second author was supported by Contract 2-CEx06-11-20 of
the Romanian Ministry of Education and Research.

\section{Monads and the category of their (bi)modules}
\label{sec:bimod}

Throughout the paper, we use the notations introduced in
\cite{BohmStefan}. That is, in the $2$-category \textrm{CAT }horizontal
composition (of functors) is denoted by juxtaposition, while $\circ $ is
used for vertical composition (of natural transformations).
For example, for two functors $F:{\mathcal C}\to {\mathcal C}'$, $G:{\mathcal
C}'\to {\mathcal C}''$ and an object $X$ in ${\mathcal C}$, instead of
$G(F(X))$ we write $GFX$.
For two natural transformations $\mu:F\rightarrow F^{\prime }$ and $
\nu:G\rightarrow G^{\prime }$ we write $G^{\prime }\mu X\circ \nu
FX:GFX\rightarrow G^{\prime }F^{\prime }X$ instead of $G^{\prime }(\mu_
X)\circ \nu_{F(X)}.$ In equalities of natural transformations
we shall omit the object $X$ in our formulae.

We shall also use a graphical representation of morphisms in a category.
For functors $F_1,\dots,F_n$, $G_1,\dots,G_m$, which can be composed to
$F_1F_2\dots F_n:{\mathcal D}_1\to {\mathcal C}$ and $G_1G_2\dots
G_m:{\mathcal D}_2\to {\mathcal C}$, and objects $X$ in ${\mathcal D}_1$ and
$Y$ in ${\mathcal D}_2$,
a morphism $
f:F_1F_2\dots F_n X\to G_1G_2\dots G_mY$ will be represented vertically,
with the domain up, as in Figure~\ref{Fig:morphisms in C}(a).
Furthermore, for a functor $T:\mathcal{C}\to\mathcal{C}'$, the morphism $Tf$
will be drawn as in (b). Keeping the notation
from the first paragraph of this section, the picture representing $\mu GX$
is shown in diagram (c). The composition $g\circ f$ of the morphisms $f:X\to
Y$ and $g:Y\to Z$ will be represented as in diagram (d).
For the multiplication $t$ and the unit $\tau$ of a monad $T$ on $\mathcal{C}$
(see Definition~\ref {de:monad}), and an object $X$ in $\mathcal{C}$, to draw
$tX$ and $\tau X$ we shall use the diagrams (e) and (f), while for
a distributive law $\mathfrak{l}:RT\to TR$ (see
Definition~\ref{de:distributive law}) $\mathfrak{l} X$
will be drawn as in the picture (g). If $\mathfrak{l}$ is invertible, the
representation of $\mathfrak{l}^{-1}X$ is shown in diagram (h).
\begin{figure}[h]
\begin{center}
{\includegraphics[scale=.78]{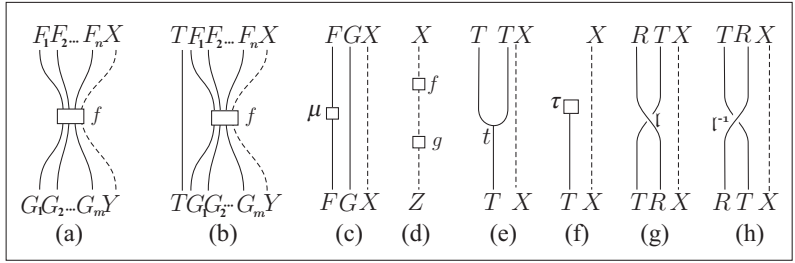}}
\end{center}
\caption{Diagrammatic representation of morphisms in a category}
\label{Fig:morphisms in C}
\end{figure}

For simplifying the diagrams containing only natural transformations, we
shall always omit the last string, that corresponds to an object in the
category.

\begin{definition}
\label{de:monad} A {\em monad} on a category $\mathcal{C}$ is a triple
$(R,r,\rho ),$
where $ R:\mathcal{C}\rightarrow \mathcal{C}$ is a functor,
$r:R^{2}\rightarrow R$ and $\rho :\mathrm{Id}_{\mathcal{C}}\rightarrow R$ are
natural transformations such that the first two diagrams in
Figure~\ref{fig:monad}, expressing associativity and unitality, are
commutative.
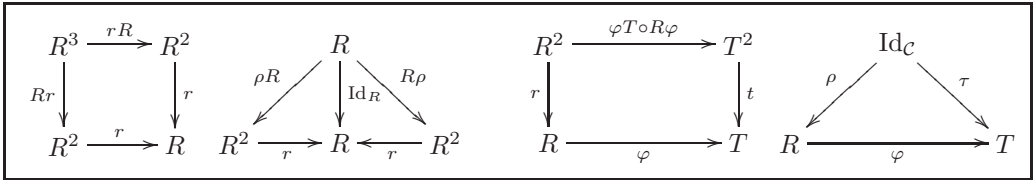
\begin{figure}[h]
\begin{center}
\fbox{
\xymatrix{
  R^3 \ar[d]_{Rr} \ar[r]^{rR} & R^2 \ar[d]^{r} \\
  R^2\ar[r]^{r} &  R }\
\xymatrix{ &  R \ar[d]^{\mathrm{Id}_R} \ar[dl]_{\rho R} \ar[dr]^{R\rho }  & \\
 R^2\ar[r]_{r}&  R   & R^2 \ar[l]^{r}            } \qquad
\xymatrix{
  R^2 \ar[d]_{r} \ar[rr]^{\varphi T\circ R\varphi} && T^2 \ar[d]^{t} \\
  R\ar[rr]_{\varphi} &&  T }\
\xymatrix{
                &\mathrm{Id}_{\mathcal{C}} \ar[ld]_{\rho }\ar[dr]^{\tau }
  \\
 R  \ar[rr]_{\mathrm{\varphi}} & &     T       }}
\end{center}
\caption{Monads and morphisms of monads.}
\label{fig:monad}
\end{figure}
\noindent We call $r$ and $\rho $ the \emph{multiplication} and the \emph{
unit} of the monad $R$, respectively.

For two monads $(R,r,\rho)$ and $(T,t,\tau )$ on $\mathcal{C}$, we say that a
natural
transformation $\varphi :R\rightarrow T$ is a \emph{morphism of monads} if
the last two diagrams in Figure~\ref{fig:monad} are commutative.
\end{definition}

\begin{example}\label{ex:k_alg}
Let $(\mathcal{C}, \otimes,\alpha,\iota_l,\iota_r,{\bf 1})$ be a
monoidal category with \emph{tensor product} $\otimes
:\mathcal{C}\times \mathcal{C} \rightarrow \mathcal{C}$ and
\emph{unit object} $\mathbf{1}\in \mathrm{OB}(\mathcal{C})$.
Recall that the associativity constraint $\alpha_{X,Y,Z}:(X\otimes
Y)\otimes Z\rightarrow X\otimes (Y\otimes Z)$ and the unit
constraints $\iota_r {X}:X\otimes \mathbf{1}\rightarrow X$ and
$\iota_l {X}:\mathbf{1}\otimes X\rightarrow X$ are natural
isomorphisms that obey \emph{Pentagon Axiom} and \emph{Triangle
Axioms}, cf. \cite[Chapter XI]{Ka}.

Algebras in monoidal categories can be defined as in the classical
case, of algebras over a commutative ring, see \cite[\S 1]{AMS}.
To such an algebra ${\bf R}$ in $\mathcal{C}$, with
multiplication $\mathbf{r}:{\bf R}\otimes {\bf R}\to {\bf R}$ and
unit $\boldsymbol{\rho}:{\bf 1}\to {\bf R}$, one associates a
monad $R:= {\bf R} \otimes (-)$ on $\mathcal{C}$. Its
multiplication and unit are respectively defined by the morphisms
$r X:=(\mathbf{r}\otimes X)\circ \alpha_{{\bf R,R},X}^{-1}$ and
$\rho X:=\boldsymbol{\rho}\otimes X$, where $X$ is an arbitrary
object in $\mathcal{C}$. An algebra homomorphism
$\boldsymbol{\varphi}:{\bf R}\to {\bf T}$ in $\mathcal{C}$ induces
a monad morphism $\boldsymbol{\varphi}\otimes (-) : {\bf R}\otimes
(-) \to {\bf T} \otimes (-)$ that will be denoted by $\varphi$.
\end{example}

\begin{definition}
For a monad $(R,r,\rho)$ on $\mathcal{C}$, a pair $(X,x)$ is said to be an
$R$-{\em module} if $X$ is an object in $\mathcal{C}$ and $x:RX\rightarrow X$
is a morphism such that the first two diagrams in Figure~ \ref{fig:module},
expressing associativity and unitality, are commutative. A
{\em morphism of $R$-modules} from $(X,x)$ to $(Y,y)$ is a morphism
$f:X\rightarrow Y$ in $\mathcal{C}$ such that the third diagram in
Figure~\ref{fig:module} is commutative. The category of $R$-modules is
denoted by $_{R}\mathcal{C}.$
\begin{figure}[h]
\begin{center}
\fbox{
\xymatrix{
  R^2X \ar[d]_{rX} \ar[r]^{Rx} & RX \ar[d]^{x} \\
  RX\ar[r]^{x} &  X }\qquad\qquad
\xymatrix{
                &RX \ar[dr]^{x}             \\
 X \ar[ru]^{\rho X} \ar[rr]_{\mathrm{Id}_X} & &     X        }\qquad\qquad
\xymatrix{
  RX \ar[d]_{Rf} \ar[r]^{x} & X \ar[d]^{f} \\
  RY\ar[r]^{y} &  Y } }
\end{center}
\caption{Modules over a monad and morphisms of modules. }
\label{fig:module}
\end{figure}
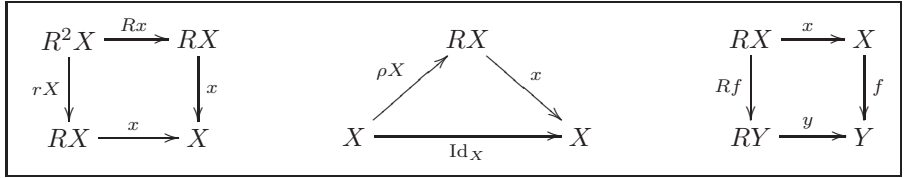
\end{definition}

\begin{example}
For a monad $R={\bf R}\otimes (-)$, induced by an algebra ${\bf
R}$ in a monoidal category ${\mathcal C}$ as in Example \ref{ex:k_alg}, ${}_R
\mathcal{C}$ is isomorphic to the category ${\bf
R}$-$\mathrm{Mod}$ of left modules for the algebra ${\bf R}$. For
the definition of modules over an algebra in a monoidal category,
see \cite[\S 1]{AMS}.
\end{example}

For a monad $(R,r,\rho)$ on a category $\mathcal{C}$, there is a faithful {\em
  forgetful functor} $\rho^\ast:{}_R\mathcal{C}\to \mathcal{C}$. It is given
  by the object map $(X,x)\mapsto X$ and it acts on the morphisms as the
  identity map. The forgetful functor possesses a left adjoint, the {\em free
  functor} $\rho_\ast:\mathcal{C}\to {}_R\mathcal{C}$. It has the object map
  $X\mapsto (RX,rX)$ and it acts on the morphisms by $f\mapsto Rf$.

Assuming that coequalizers exist in $\mathcal{C}$, it is natural to ask if
the category of modules over a monad $T$ on $\mathcal{C}$ has the same
property.
In what follows we will frequently use the standard result that this question
has a positive answer if $T$
preserves coequalizers in $\mathcal{C}$.
Recall that a functor  $F:\mathcal{C}\rightarrow \mathcal{D}$ preserves
coequalizers if, for any coequalizer $(Z,\pi )$ in $\mathcal{C}$
\begin{equation}  \label{dia:fork}
\begin{xy} (0,0)*+{X}="v7";(20,0)*+{Y}="v8";(40,0)*+{Z,}="v9";
{\ar@<1mm>^-{f} "v7";"v8" }; {\ar@<-.5mm>_-{g} "v7";"v8"}; {\ar@{->}^{ \pi}
"v8"; "v9"}; \end{xy}
\end{equation}
$(FZ,F\pi)$ is the coequalizer of $(Ff,Fg)$.
Note that in this case the canonical morphisms $\pi$ and $F\pi$ are
epimorphisms (while for an arbitrary epimorphism $p$ in $\mathcal{C}$, the
morphism $Fp$
is not necessarily an epimorphism in $\mathcal{D}$).

\begin{proposition}
\label{le:coequalizer}
Let $(T,t,\tau)$ be a monad on a category $\mathcal{C}$ that preserves
coequalizers.
Assume that $(Z,\pi)$ is the coequalizer of a parallel pair of morphisms $
(f,g)$ in $\mathcal{C}$, as in (\ref{dia:fork}). If, in addition, $X$ and $Y$
are $T$-modules
such that $f$ and $g$ are morphisms of $T$-modules, then there is a
unique $T$-module structure on $Z$ such that $(Z,\pi )$ is the coequalizer
of $(f,g)$ in $_{T} \mathcal{C}.$ In particular, if any pair of parallel
morphisms in $\mathcal{C}$ has a coequalizer then
any pair of parallel morphisms in ${}_T \mathcal{C}$ has a coequalizer, too.
\end{proposition}

\begin{proof}
Let $x$ and $y$ denote the actions of $T$ on $X$ and $Y$, respectively. Since
$(TZ,T\pi )$
is the coequalizer of $(Tf,Tg)$ in $\mathcal{C}$, and $\pi \circ y$
coequalizes $(Tf,Tg)$, it
follows that there is a unique morphism $z:TZ\rightarrow Z$ such that
\begin{equation}
z\circ T\pi =\pi \circ y.  \label{ec:z}
\end{equation}
Our aim is to prove that $(Z,z)$ is a $T$-module.
Since $T$ is a functor, in view of (\ref{ec:z}), associativity of $y$ and of
the fact that $t$ is natural, we get
\begin{equation*}
z\circ Tz\circ T T\pi=
z\circ T\pi \circ Ty=\pi \circ
y\circ Ty=\pi \circ y\circ tY=z\circ T\pi \circ tY=z\circ tZ\circ TT\pi .
\end{equation*}
Since $T$ preserves coequalizers, $TT\pi$ is an epimorphism, so the
associativity 
condition in the definition of $T$-modules is verified. In a similar way one
can show that $\tau Z$ is a right inverse of $z$. Thus $(Z,z)$ is a $T$
-module and, in view of (\ref{ec:z}), $\pi $ is a morphism of $T$-modules.

Let $h$ be a morphism in $_{T}\mathcal{C}$ that coequalizes $(f,g).$ As $
(Z,\pi)$ is the coequalizer of $(f,g)$ in $\mathcal{C}$, there is a
unique morphism $\overline{h}$ in $\mathcal{C}$ such that $\overline{ h}
\circ \pi =h.$ Obviously, $\overline{h}$ is a morphism of $T$-modules, so
the proposition is proved.
\end{proof}

Let $\varphi:R\to T$ be a morphism of monads on $\mathcal{C}$. We show that
one can associate to $\varphi$ ``forgetful'' and ``free'' functors connecting
the categories of modules over $R$ and $T$.
The construction of $\varphi ^{\ast }:{}_{T}\mathcal{C}\rightarrow
_{R}\mathcal{C}$ is quite obvious, hence the proof is left to the reader.

\begin{proposition}
\label{pr:phi_*} There exists a functor $\varphi ^{\ast }:{}_{T}\mathcal{C}
\rightarrow {}_{R}\mathcal{C}$ such that
\begin{equation}
\varphi ^{\ast }(X,x):=(X,x\circ \varphi X)  \label{ec:phi_*}
\end{equation}
and $\varphi ^{\ast }f=f,$ for every $T$-module $(X,x)$ and every $T$-module
morphism $f$.
\end{proposition}

\begin{proposition}
\label{pr:adjoint functors}Let $\mathcal{C}$ be a category with
coequalizers. If $(R,r,\rho)$ and $(T,t,\tau)$ are monads on $\mathcal{C}$,
such that $T$ that preserves coequalizers, and $\varphi:R\to T$ is a morphism
of monads then the functor $\varphi^{\ast }$ in Proposition \ref{pr:phi_*} has
a left adjoint.
\end{proposition}

\begin{proof}
Let $(X,x)$ be an $R$-module. It is not difficult to see that $(TRX,tRX)$
and $(TX,tX)$ are $T$-modules, and that $tX\circ T\varphi X$ and $Tx$ are
morphism of $T$-modules. Let
\begin{equation}
\left( \varphi _{\ast }(X,x),\pi X\right) :=\mathrm{Coeq}(tX\circ T\varphi
X,Tx).  \label{ec:fi_ast}
\end{equation}
By Proposition \ref {le:coequalizer}, $\varphi _{\ast }(X,x)$ is a $T$-module
such that $\pi
X:(TX,tX)\rightarrow \varphi _{\ast }(X,x)$ is a morphism of $T$-modules. By
the universal property of coequalizers, for any $R$-module morphism $
f:X\rightarrow Y,$ there is a $T$-module morphism $\varphi _{\ast }f$
such that
\begin{equation}\label{eq:pi_nat}
\varphi _{\ast }f\circ \pi X=\pi Y\circ Tf.
\end{equation}
In conclusion,
we have constructed a functor $\varphi_\ast$ from $_{R}\mathcal{C}$ to
$_{T}\mathcal{C}.$ Moreover, one can
compose the forgetful functor $\rho^\ast:{}_R\mathcal{C}\to \mathcal{C}$ with
the free functor $\tau_\ast:\mathcal{C}\to {}_T \mathcal{C}$ to obtain
a functor from $_{R}\mathcal{C}$ to $_{T}\mathcal{C}$, given by $(X,x)\mapsto
(TX,tX)$ and $f\mapsto Tf$. Therefore, (\ref{eq:pi_nat}) means that $(\pi
X)_{X\in \text{\textrm{Ob} }_{R}\mathcal{C}}$ defines a natural transformation
$\pi:\tau_\ast \rho^\ast\to \varphi _{\ast}$.

It remains to show that $\varphi _{\ast }$ is a left adjoint of $\varphi
^{\ast }.$ For objects $(X,x)$ in $_{R}\mathcal{C}$ and $(Y,y)$ in
$_{T}\mathcal{C}$ we define
\begin{equation*}
\Phi _{X,Y}:\mathrm{Hom}_{_{T}\mathcal{C}}(\varphi _{\ast
}(X,x),(Y,y))\rightarrow \mathrm{Hom}_{_{R}\mathcal{C}}((X,x),\varphi ^{\ast
}(Y,y)),\ \ \ \Phi _{X,Y}(f)=f\circ \pi X\circ \tau X.
\end{equation*}
In order to prove that $\Phi _{X,Y}(f)$ is a morphism of $R$-modules,
recall that, by the proof of Proposition~\ref{le:coequalizer}, the module
structure on $\varphi _{\ast }(X,x)$ is given by the unique morphism $
\overline{x}$ satisfying
\begin{equation}
\overline{x}\circ T\pi X=\pi X\circ tX.  \label{ec:x bar}
\end{equation}
As $\pi X\circ tX\circ T\varphi X=\pi X\circ Tx$, by a simple but tedious
computation, one can show that
\begin{equation*}
y\circ \varphi {Y}\circ R\Phi _{X,Y}(f)=\Phi _{X,Y}(f)\circ x,
\end{equation*}
i.e. $\Phi _{X,Y}(f)$ is a morphism of $R$-modules from $(X,x)$ to
$\varphi ^{\ast }(Y,y).$

We are going to construct an inverse of $\Phi _{X,Y}.$ Let $g$ be a morphism
of $R$-modules from $(X,x)$ to $\varphi ^{\ast }(Y,y)$. Then $y\circ Tg$ is
a morphism of $T$-modules and coequalizes $(Tx,tX\circ T\varphi X).$ Since $
\varphi _{\ast }(X,x)$ is the coequalizer in $_{T}\mathcal{C}$ of this pair,
there exists a unique $T$-module morphism $\Theta _{X,Y}(g)$ such that
\begin{equation*}
\Theta _{X,Y}(g)\circ \pi X=y\circ Tg.
\end{equation*}
Furthermore, by construction of $\Phi _{X,Y}$ and $\Theta _{X,Y}$, and the
facts that $\tau$ is natural and $y$ is unital, we get
\begin{equation*}
\Phi _{X,Y}\left( \Theta _{X,Y}(g)\right) =
\Theta _{X,Y}(g)\circ \pi X\circ\tau X=
y\circ Tg\circ \tau X =
y\circ \tau Y\circ g = g.
\end{equation*}
Thus we deduce that $\Theta _{X,Y}$ is a right inverse of $\Phi _{X,Y}$. Let $
f:\varphi _{\ast }(X,x)\to (Y,y)$ be a morphism of $T$-modules. To conclude,
we must prove
\begin{equation*}
\Theta _{X,Y}(\Phi _{X,Y}(f))\circ \pi X=f\circ \pi X.
\end{equation*}
The left hand side of this relation can be rewritten as
\begin{equation*}
y\circ T\Phi _{X,Y}(f)=y\circ Tf\circ T\pi X\circ T\tau X=f\circ
\overline{x}\circ T\pi X\circ T\tau X=f\circ \pi X\circ tX\circ T\tau
X=f\circ \pi X,
\end{equation*}
where for the first two equalities we used the definitions of $\Phi _{X,Y}$
and $\Theta _{X,Y}.$ Since $f$ and $\pi X$ are morphisms of $T$-modules we
obtained the third and the fourth relations, while for the last one we used
the definition of monads.
\end{proof}

\begin{example}\label{ex:phi*}
Let $(\mathcal{C,}\otimes,\alpha,\iota_l,\iota_r,\mathbf{1})$ be an abelian
monoidal category, that is $\mathcal{C}$ be abelian, and $X\otimes
(-):\mathcal{C}\rightarrow \mathcal{C}$ and $(-)\otimes
X:\mathcal{ C}\rightarrow \mathcal{C}$ be additive and right
exact functors, for any $ X\in \mathcal{C}.$ For details on
abelian monoidal categories, the reader is referred to \cite[\S
1.]{AMS}. Following \cite[(1.11)]{AMS}, for a right $\bf R$-module
$\mu_X:X\otimes \mathbf{R}\to X$ and a left $\bf R$-module
$\mu_Y:\mathbf{R}\otimes Y \to Y$ one defines $X\otimes
_{\mathbf{R}}Y$ by
\begin{equation*}
X\otimes _{\mathbf{R}}Y:=\mathrm{Coker}(\mu _{X}\otimes Y-X\otimes \mu _{Y}).
\end{equation*}
Take another algebra $\bf T$ in $\mathcal{C}$. Proceeding as in
\cite[(1.4)]{AMS} one defines the category $\boldsymbol{T}$-$\mathrm{Mod}$-$
\boldsymbol{R}$ of $\bf {T}$-$\bf {R}$ bimodules in $\mathcal{C}$. By
definition, $(X,\mu _{X}^{l},\mu _{X}^{r})$ is an ${\bf T}$-${\bf R}$ bimodule
if $\mu _{X}^{l}:{\bf T}\otimes X\rightarrow X$ and $\mu _{X}^{r}:X\otimes
{\bf R}\rightarrow X$ define compatible left and right module
structures such that
\begin{equation}
\mu _{X}^{l}\circ (T\otimes \mu _{X}^{r})=\mu _{X}^{r}\circ
(\mu _{X}^{l}\otimes R)\circ \alpha_{{\bf T},X,{\bf R}}^{-1}.
\label{ec:bimod}
\end{equation}
If  $(X,\mu_X^l,\mu_X^r)$ is a $\bf T$-$\bf R$ bimodule and
$(Y,\mu_Y^l,\mu_Y^r)$ is an $\bf R$-$\bf S$ bimodule, then
$X\otimes_{\bf R} Y$ is a $\bf T$-$\bf S$ bimodule, for any
algebra $\bf S$ in $\mathcal{C}$. Its left $\bf T$-action and
right $\bf S$-action are respectively induced by $\mu
_{X}^{l}\otimes Y$ and $X\otimes \mu _{Y}^{r}.$ As in
\cite[Theorem 1.12]{AMS}, one can prove that
$\mathbf{T}$-$\mathrm{Mod}$-$\mathbf{R}$ is abelian. Moreover,
$\mathbf{R}$-$\mathrm{Mod}$-$\mathbf{R}$ is monoidal with respect
to $ (-)\otimes _{\bf R}(-),$ see \cite[Theorem 1.12]{AMS}.

For a monad morphism $\boldsymbol{\varphi}\otimes (-): {\bf R}
\otimes (-) \to  {\bf T} \otimes (-)$, induced by an algebra
homomorphism $\boldsymbol{\varphi}:{\bf R}\to {\bf T}$, the
functors in Propositions \ref{pr:phi_*} and \ref{pr:adjoint
functors} are $\varphi^\ast:{\bf T}$-$\mathrm{Mod}\to {\bf
R}$-$\mathrm{Mod}$, the `restriction of scalars', and
$\varphi^\ast={\bf T}\otimes_{\bf R} (-):{\bf R}$-$\mathrm{Mod}
\to {\bf T}$-$\mathrm{Mod}$, the `extension of scalars'.
\end{example}

Distributive laws were introduced by J. Beck \cite{Be}. As we
shall see later, they give a way to compose two monads in order to
obtain a monad.

\begin{definition}
\label{de:distributive law} A \emph{distributive law} between two monads
$(R,r,\rho )$ and $(T,t,\tau ) $ is a natural transformation $\mathfrak{l}
:RT\rightarrow TR$ satisfying the four conditions in Figure~\ref{fig:DLaw}.
\begin{figure}[h]
\begin{center}
{\includegraphics[scale=.78]{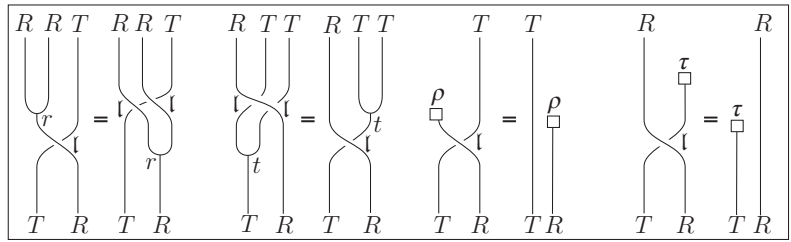}}
\end{center}
\caption{The definition of distributive laws.}
\label{fig:DLaw}
\end{figure}
\end{definition}

\begin{remark}
Since we are using for the first time the diagrammatic representation of
morphisms, we note that the first and the third relations in Figure~\ref
{fig:DLaw} can be explicitly written as follows:
\begin{equation*}
\mathfrak{l}\circ rT =Tr\circ \mathfrak{l} R\circ R\mathfrak{l},\qquad\qquad
\mathfrak{l}\circ\rho T = T\rho.
\end{equation*}
\end{remark}

\begin{example}\label{ex:lrt}
Let $\mathbf{R}$ and $\mathbf{T}$ be two algebras in a braided
monoidal category $(\mathcal{C},
\otimes,\alpha,\iota_l,\iota_r,{\bf 1},\chi)$, with braiding
$\chi:\mathcal{C}\times\mathcal{C}\to
\mathcal{C}\times\mathcal{C}$. Let $R={\bf R}\otimes (-)$ and
$T={\bf T}\otimes (-)$ be the monads induced by ${\bf R}$ and
${\bf T}$, as in Example \ref{ex:k_alg}. For an object $X$ in $\mathcal{C}$,
we put
\begin{align*}
\mathfrak{l}X& :=\alpha_{\mathbf{T},\mathbf{R},X}\circ \left( \chi_{\mathbf{R},
\mathbf{T}}\otimes X\right) \circ \alpha_{\mathbf{R},\mathbf{T},X}^{-1}, \\
\mathfrak{r}X& :=\alpha_{\mathbf{R},\mathbf{R},X}\circ \left( \chi_{\mathbf{R},
\mathbf{R}}\otimes X\right) \circ \alpha_{\mathbf{R},\mathbf{R},X}^{-1}, \\
\mathfrak{t}X& :=\alpha_{\mathbf{T},\mathbf{T},X}\circ \left( \chi_{\mathbf{T},
\mathbf{T}}\otimes X\right) \circ
\alpha_{\mathbf{T},\mathbf{T},X}^{-1}.
\end{align*}
One can see easily that $\mathfrak{l}:RT\to TR$,
$\mathfrak{r}:RR\to RR$ and $\mathfrak{t}:TT\to TT$ are
distributive laws.
\end{example}

\begin{definition}
\label{def:BD-law} Consider a monad $(T,t,\tau)$ on a category
$\mathcal{C}$. A distributive law $\mathfrak{t}:TT\rightarrow TT$ is
said to be a {\em BD-law} if it satisfies the \emph{YB}-equation:
\begin{equation*}
\mathfrak{t}T\circ T\mathfrak{t}\circ \mathfrak{t}T=T\mathfrak{t}\circ
\mathfrak{t}T\circ T\mathfrak{t}.
\end{equation*}
For a diagrammatic representation of the \emph{YB}-equation see the first
picture in Figure~\ref{fig:Septuple}.
\end{definition}

The distributive laws $\mathfrak{t}$ and $\mathfrak{r}$ in Example
\ref{ex:lrt} are, in fact, \emph{BD}-laws.

\section{Braided pairs of monads and their morphisms}
\label{sec:construction}

Our starting point of a construction of para-cocyclic objects will be the
following notion.

\begin{definition}
\label{de:septuple}
The sextuple $\mathcal{S}:=(\mathcal{C},T,R,\mathfrak{l},\mathfrak{t},
\mathfrak{r})$ is said to be a
\emph{\bpm} if
and only if the following conditions are satisfied.
\begin{itemize}
\item $\mathcal{C}$ is a category, in which any pair of parallel morphisms has
  a coequalizer.

\item $(R,r,\rho)$ and $(T,t,\tau)$ are monads on $\mathcal{C}$, both of which
  preserve coequalizers.

\item $\mathfrak{l}:RT\rightarrow TR$ is an {\em  invertible} distributive law.

\item $\mathfrak{t}:TT\rightarrow TT$ and $\mathfrak{r}:RR\to RR$ are
  {\em invertible} BD-laws such that
conditions (A~1) and (A~2) in Figure~\ref{fig:Septuple} hold.
\end{itemize}
\begin{figure}[h]
\begin{center}
\includegraphics[scale=.78]{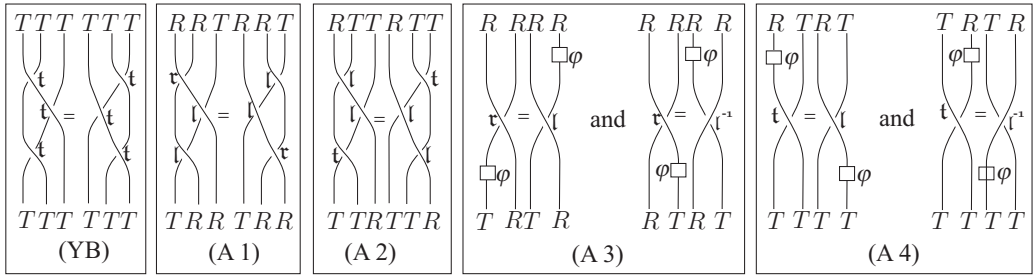}
\end{center}
\caption{\emph{YB}-equation and conditions (A~1)--(A~4).}
\label{fig:Septuple}
\end{figure}

A morphism of monads $\varphi :R\rightarrow T$ is said to be {\em
\lb} if it satisfies conditions (A~3) and (A~4) in
Figure~\ref{fig:Septuple}.
\end{definition}

\begin{remark}
In the proof of several results we do not need all assumptions
from the definition of a \lb monad morphism in Definition
\ref{de:septuple}. Still, for reader's convenience, we prefer to
state all of them at the same time.
\end{remark}

\begin{example}
\label{ex:alg_sep} Let $\boldsymbol{\varphi}:(\mathbf{R},\mathbf{r},
\boldsymbol{\rho})\rightarrow
(\mathbf{T},\mathbf{t},\boldsymbol{\tau})$
be a morphism of algebras in a braided monoidal category $(\mathcal{C,}
\otimes ,\alpha ,\iota _{l},\iota _{r},\mathbf{1},\chi )$. We take
the monads $R$ and $T$ as in Example \ref{ex:k_alg}, and the
distributivity laws
$\mathfrak{l}$, $\mathfrak{r}$ and $\mathfrak{t}$ as in Example \ref{ex:lrt}
. Clearly,
$(\mathcal{C},R,T,\mathfrak{l},\mathfrak{r},\mathfrak{t})$ is a
braided pair of monads. Moreover, by pushing $\boldsymbol{\varphi}$ over $
\mathfrak{r}$ in the second diagram in (A~3) and under
$\mathfrak{t}$ in the
second picture in (A~4), it follows that $\varphi :=\boldsymbol{\varphi}
\otimes (-)$ is a braid preserving monad morphism if, and only if,
\begin{equation}\label{eq:br_ex}
(\boldsymbol{\varphi}\otimes \mathbf{T})\circ \chi _{
\mathbf{T,R}} =(\boldsymbol{\varphi}\otimes
\mathbf{T})\circ \chi _{\mathbf{R},\mathbf{T}}^{-1}, 
\qquad \textrm{and}\qquad
\chi _{\mathbf{R,T}}\circ (\mathbf{R}\otimes  \boldsymbol{\varphi}) =
\chi_{\mathbf{T,R}}^{-1}\circ (\mathbf{R}\otimes
\boldsymbol{\varphi}).
\end{equation}
For example, if $\chi _{\mathbf{T,R}}=\chi_{\mathbf{R},\mathbf{T}}^{-1}$,
then $\varphi $ is such a morphism. In a similar manner one can show that $
\varphi $ is a braid preserving monad morphism, provided that $\chi _{
\mathbf{R,R}}=\chi _{\mathbf{R},\mathbf{R}}^{-1}$ and $\chi _{\mathbf{T,T}
}=\chi _{\mathbf{T},\mathbf{T}}^{-1}$. For, it is enough to
rewrite the second equalities in (A~3) and (A~4) by pushing
$\varphi $ under $\mathfrak{l }^{-1}.$ 
\end{example}

Distributive laws, as we have already remarked, were introduced by Beck in
order to construct a monad on the composite of two monads.
For two monads $(R,r,\rho)$ and $(T,t,\tau)$, a
distributive law
$\mathfrak{l}:RT \to TR$ induces a monad $TR$ on $\mathcal{C},$ with
multiplication and unit 
\begin{equation*}
Tr\circ tR^{2}\circ T\mathfrak{l}R=tR\circ T^2r\circ T\mathfrak{l}R:TRTR\to TR
\quad\textrm{and}\quad
T\rho\circ \tau =\tau R \circ\rho :\mathrm{Id}_{\mathcal{C}}\to TR.
\end{equation*}
Therefore, it makes sense to speak about $_{TR}\mathcal{C}$, the category $
TR $-modules. Its objects can be described equivalently as follows. Let $
(X,x)$ be an object in $_{TR} \mathcal{C}.$ We set $x_{T}:=x\circ T\rho X$ and
$x_{R}:=x\circ \tau RX.$ One can prove that $(X,x_R)$ is an $R$-module, $
(X,x_T)$ is a $T$-module and these structures commute in the sense that
\begin{equation}
x_{T}\circ Tx_{R}\circ \mathfrak{l}X=x_{R}\circ Rx_{T}.  \label{ec:TR module}
\end{equation}
Conversely, if $(X,x_{R})$ is an $R$-module and $(X,x_{T})$ is a $T$-module
that obey relation (\ref{ec:TR module}) then $X$ is an $TR$-module with
respect to $x:=x_{T}\circ Tx_{R}.$ Moreover, a morphism in $\mathcal{C}$ is
a morphism of $TR$-modules if, and only if, it is a morphism of $R$-modules
and $T$-modules with respect to the above defined structures. Thus, $_{TR}
\mathcal{C}$ is isomorphic to the category of triples $(X,x_{R},x_{T})$ such
that $x_{R}$ and $x_{T}$ define commuting module structures over $R$ and $T,$
respectively. For details see \cite{KLV}.

For any  \emph{BD}-law  $\mathfrak{r}:RR \to RR$, one can deform the monad
$(R,r,\rho )$ to obtain a new monad
\begin{equation*}
R^{\mathfrak{r}}:=R,\qquad \qquad r^{\mathfrak{r}}:=r\circ \mathfrak{r},\qquad
\qquad \rho ^\mathfrak{r}{}:=\rho .
\end{equation*}
For a
\bpm
$\mathcal{S}:=(\mathcal{C},T,R,\mathfrak{l},\mathfrak{t},\mathfrak{r})$ and a
\lb monad morphism $\varphi:R\to T$,
it follows by the definition of distributive laws and (A~1) that
$\mathfrak{l}:R^{\mathfrak{r}}T\rightarrow TR^{\mathfrak{r}}$ is a
distributive law too.
Hence $TR^{\mathfrak{r}}$ admits a monad structure. By the above
considerations, the category of  $TR^{\mathfrak{r}}$-modules is isomorphic to
the category of triples $(X,x_R^0,x_T)$, where
$(X,x_R^0)$ is an $R^{\mathfrak{r}}$-module and $(X,x_T)$ is a
$T$-module such that the compatibility condition
\begin{equation}\label{eq:T-R_bim}
x_{T}\circ Tx_{R}^0\circ
\mathfrak{l}X=x_{R}^0\circ
    R^{\mathfrak{r}}x_{T}
\end{equation}
holds. Their morphisms are both $T$-module, and $R^{\mathfrak{r}}$-module
morphisms. We call the triple $(X,x_R^0,x_T)$ a
$(T,R)$\emph{-bimodule} and denote their category by ${}_T \mathcal{C}_R$.

\begin{remark}\label{rem:R_op}
By the above considerations, in particular, for any monad $R$ and
{\em BD}-law $\mathfrak{r}:RR\to RR$,
also $\mathfrak{r}:$ $R^{\mathfrak{r}}R\rightarrow
RR^{\mathfrak{r}}$ is a distributive law. Hence, as above,
$R^{e}:=RR^{\mathfrak{r}}$ is a monad, that is called the \emph{enveloping
monad of }$R.$
The category ${}_{R^e}\mathcal{C}$ is isomorphic to the category
of $(R,R)$-bimodules. Recall that an $(R,R)$-bimodule is an
object $X$ in $\mathcal{C}$, together with two morphisms $x:RX\rightarrow X$
and $ x^0:R^{\mathfrak{r}}X\rightarrow X$ such that $(X,x)$ is an
$R$-module, $(X,x^0)$ is an $R^{\mathfrak{r}}$-module and these
structures commute in the sense that
\begin{equation}\label{eq:R-bimod}
x\circ Rx^0\circ \mathfrak{r}X=x^0\circ R^{\mathfrak{r}}x.
\end{equation}
The category of $(R,R)$-bimodules will be denoted by ${}_R \mathcal{C}_R$. In
order to simplify notations, the forgetful functor $U:{}_R \mathcal{C}_R \to
\mathcal{C}$ (with object map $(X,x,x^0)\mapsto X$) will be omitted in our
formulae whenever it causes no danger.

For an \emph{invertible} \emph{BD}-law  $\mathfrak{r}:RR \to RR$, also
$\mathfrak{r}^{-1}: R^{\mathfrak{r}} R^{\mathfrak{r}} \to R^{\mathfrak{r}}
R^{\mathfrak{r}}$ is an (invertible) \emph{BD}-law. Hence the previous
construction can be repeated with $\mathfrak{r}:RR \to RR$ replaced by
$\mathfrak{r}^{-1}: R^{\mathfrak{r}} R^{\mathfrak{r}} \to R^{\mathfrak{r}}
R^{\mathfrak{r}}$. The objects of the resulting category
${}_{R^{\mathfrak{r}}}\mathcal{C}_{R^{\mathfrak{r}}}$ are triples $(X,x^0,x)$,
where $(X,x^0)$ and $(X,x)$ are $R^{\mathfrak{r}}$ and
$R=(R^{\mathfrak{r}})^{\mathfrak{r}^{-1}}$-modules, respectively, such that
\begin{equation*}
x^0\circ R^{\mathfrak{r}}x\circ \mathfrak{r}^{-1}X=x\circ
Rx^0.
\end{equation*}
Hence the object map $(X,x,x^0)\mapsto (X,x^0,x)$
induces an isomorphism ${}_R \mathcal{C}_R\to
{}_{R^{\mathfrak{r}}}\mathcal{C}_{R^{\mathfrak{r}}}$.
\end{remark}

\begin{example}\label{ex:rcr}
Let $R$ and $T$ be the monads coming from  algebras ${\bf R}$ and
${\bf T}$ in a  braided monoidal $\mathcal{C}$, as in Example
\ref{ex:k_alg}. We denote by $\mathfrak{l}$, $\mathfrak{r}$ and
$\mathfrak{t}$ the distributive laws defined in Example
\ref{ex:lrt}. In this case,
the category ${}_T \mathcal{C}_R$ is isomorphic to the category ${\bf
T}$-$\mathrm{Mod}$-${\bf R}$, defined in Example \ref{ex:phi*}.
Indeed, an object in $_{T}\mathcal{C}_{R}$ is a triple
$(X,x,x^{0}),$ such that $(X,x)$ is a $T$-module, $(X,x^{0})$ is
an $R^{\mathfrak{r}}$-module and these structures satisfy
relation (\ref{eq:T-R_bim}). We define a functor
\begin{equation*}
F:{}_{T}\mathcal{C}_{R}\rightarrow \mathbf{T}\text{-}\mathrm{Mod}\text{-}
\mathbf{R,\qquad }F(X,x,x^{0}):=(X,x,x^{0}\circ
\chi_{\mathbf{R},X}^{-1}).
\end{equation*}
Let us first prove that $F$ is well defined, that is $(X,x,x^{0}\circ \chi_{
\mathbf{R},X}^{-1})$ is an object in
$\mathbf{T}$-$\mathrm{Mod}$-$\mathbf{R.}
$ We have to check that $(X,x^{0}\circ \chi_{\mathbf{R},X}^{-1})$ is a right $
\mathbf{R}$-module and that the structure maps obey (\ref{ec:bimod}).
\begin{figure}[t]
\begin{center}
{\includegraphics[scale=.78]{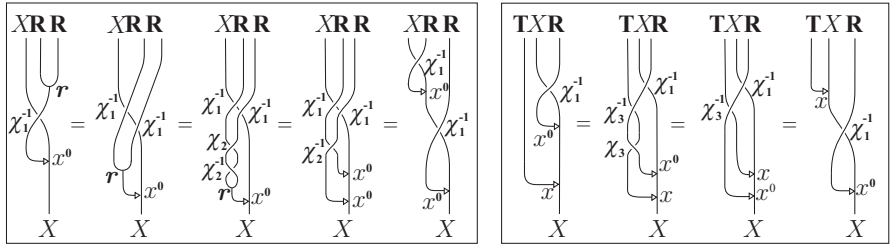}}
\end{center}
\caption{The categories $_{T}\mathcal{C}_{R}$ and $\mathbf{T}$-$\mathrm{Mod}$
-$\mathbf{R}$ are equivalent.}\label{Fig:echiv}
\end{figure}
The graphical proofs of these properties are given in Figure \ref{Fig:echiv}
, where for simplicity we use the notation $\chi_{1}:=\chi_{\mathbf{R},X}$, $
\chi_{2}:=\chi_{\mathbf{R},\mathbf{R}}$ and
$\chi_{3}:=\chi_{\mathbf{R},\mathbf{T}}$. In the first frame we
prove that $X$ is a right $\mathbf{R}$-module. The first and the
fourth equalities follow by the fact that the braiding $\chi$ is
a natural transformation. The second
identity is trivial, while for the third one we used that $(X,x^{0})$ is an
$\mathbf{R}^{\mathfrak{r}}$-module. In the second frame we show that the
module structures of $X$ satisfy condition (\ref{ec:bimod}). The
first relation in that frame is obvious. The second one follows by
(\ref{eq:T-R_bim}) and the last one is a consequence of naturality
of the braiding.
Obviously, $F^{-1}:\mathbf{T}$-$\mathrm{Mod}$-$\mathbf{R}
{}\rightarrow {}_{T}\mathcal{C}_{R}$ is given by $F^{-1}(X,\mu
_{X}^{l},\mu _{X}^{r}):=(X,\mu _{X}^{l},\mu _{X}^{r}\circ
\chi_{\mathbf{R},X})$.

In view of this example, from now on, we shall always regard
an object in $_T\mathcal{C}_R$ as a $\bf T$-$\bf R$ bimodule.
\end{example}

\begin{proposition}\label{prop:RX-mod}
Let $(R,r,\rho)$ be a monad and $\mathfrak{r}:RR\to RR$ be a
\emph{BD}-law. For an $(R,R)$-bimodule $(X,x,x^0)$, also the
triple $(RX,Rx\circ \mathfrak{r}X,Rx^{0}\circ \mathfrak{r}X)$  is
an $R$-bimodule. This construction defines a lifting ${\widetilde
  R}:{}_{R}\mathcal{C}_R \rightarrow {}_{R}\mathcal{C}_{R}$ of $R$.
\end{proposition}
\begin{proof}
Let us denote the actions of $R$ and $R^{\mathfrak{r}}$ on $RX$ by
$y$ and $y^{0}$, respectively. In the first frame of
Figure~\ref{fig:R_module} we prove that $RX$ is an $R$-module. The
first equality follows by the definition of $y$. The second
equality is a consequence of the definition of distributive laws,
cf. Figure~\ref{fig:DLaw}, while the third one results by the
definition of $R$-modules. For the last relation we use the
definition of $y$.
\begin{figure}[h]
\begin{center}
\includegraphics[scale=.78]{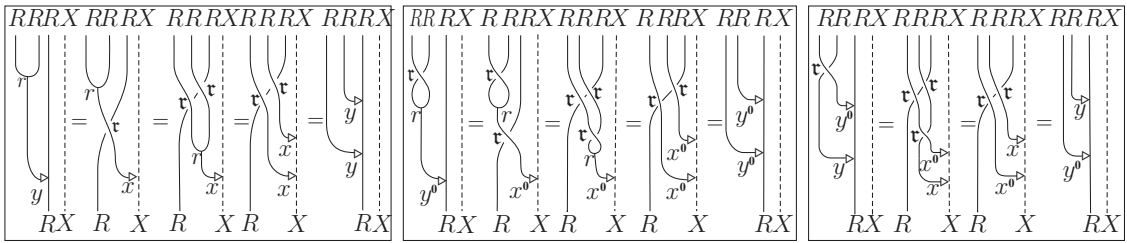}
\end{center}
\caption{$RX$ is an $(R,R)$-bimodule.} \label{fig:R_module}
\end{figure}
In the second frame we prove that $RX$ is an
$R^{\mathfrak{r}}$-module too. We proceed as above using, in
addition, \emph{YB}-equation. The fact that $y$ and $y^{0}$
commute, i.e. they satisfy relation (\ref{eq:R-bimod}), is proved
in the third frame, where for the first relation we use the
definition of the actions and \emph{YB}-equation. Since $x$ and
$x^{0}$ commute, we obtain the second relation. We conclude the
proof of the proposition by applying the definition of $y$ and
$y^{0}$.
\end{proof}

\begin{remark}
If $\mathfrak{r}:RR\to RR$ is an invertible \emph{BD}-law,
Proposition \ref{prop:RX-mod} can be applied to the
\emph{BD}-law $\mathfrak{r}^{-1}: R^{\mathfrak{r}} R^{\mathfrak{r}} \to
R^{\mathfrak{r}} R^{\mathfrak{r}}$ and the $R^{\mathfrak{r}}$-bimodule
$(X,x^0,x)$. Using the isomorphism ${}_{R}\mathcal{C}_R\cong
{}_{R^{\mathfrak{r}}}\mathcal{C}_{R^{\mathfrak{r}}}$ in Remark \ref{rem:R_op},
we obtain a functor ${\widetilde R}^0:
{}_{R}\mathcal{C}_R \rightarrow {}_{R}\mathcal{C}_{R}$.
\end{remark}

\begin{example}
For a monad $R={\bf R}\otimes (-)$, induced by an algebra ${\bf
R}$ as in Example \ref{ex:k_alg}, the functor ${\widetilde R}:{\bf
R}$-$\mathrm{Mod}$-${\bf R} \to {\bf R}$-$\mathrm{Mod}$-${\bf R}$,
maps an ${\bf R}$-bimodule $(X,\mu_X^l,\mu_X^r)$ to ${\bf
R}\otimes X$, with ${\bf R}$-actions
\begin{align}
{\mu}_{\mathbf{R}\otimes X}^l&:=({\bf R}\otimes
\mu_X^l)\circ\alpha_{{\bf R,R},X}\circ(\chi_{\bf R,R}\otimes
X)\circ\alpha_{{\bf R,R},X}^{-1},\qquad {\mu}_{\mathbf{R}\otimes
X}^{r}:=({\bf R}\otimes \mu_X^r)\circ\alpha_{{\bf R},X,{\bf R}}.
\end{align}
\end{example}

\begin{remark} \label{rem:op}
Note that, for a \bpm $\mathcal{S}:=(\mathcal{C},T,R,\mathfrak{l},
\mathfrak{t}, \mathfrak{r})$ and  any \lb monad morphism
$\varphi:R \to T$ is \lb also for the \bpm
$\mathcal{S}^{0}:=(\mathcal{C},
T^{\mathfrak{t}^{-1}},R^{\mathfrak{r}},
\mathfrak{l},\mathfrak{t}^{-1}, \mathfrak{r}^{-1})$.
\end{remark}

\begin{proposition}\label{prop:i-o_mod}
Consider a
\bpm $\mathcal{S}:=(\mathcal{C},T,R ,\mathfrak{l},\mathfrak{t},\mathfrak{r})$
and a
\lb monad morphism $\varphi:R \to T$. For an $R$-module $(X,x)$ and an
$R^{\mathfrak r}$-module $(Y,y^0)$, the following assertions hold.
\begin{enumerate}
\item The triple $(TX,x_{i},x_{i}^0)$ defines an $(R,R)$-bimodule,
where
\begin{equation}
x_{i}:=Tx\circ \mathfrak{l}X\qquad \text{and\qquad }x_{i}^0:=
tX\circ T\varphi X\circ \mathfrak{l}X.  \label{ec: T_i}
\end{equation}
It is called the \emph{inside structure} on $TX$. This
construction defines a functor ${}_{R}\mathcal{C}\rightarrow
{}_{R}\mathcal{C}_{R}$. \item The triple $(TY,y_{o},y_{o}^0)$
defines an $(R,R)$-bimodule, where
\begin{equation}\label{ec: T_o}
y_{o}:=tY\circ \varphi TY\text{ \qquad and\qquad }
y_{o}^0:=Ty^0\circ \mathfrak{l}Y.
\end{equation}
It is called the \emph{outside structure} on $TY$. This
construction defines a functor
${}_{R^{\mathfrak{r}}}\mathcal{C}\rightarrow
{}_{R}\mathcal{C}_{R}$.
\end{enumerate}
\end{proposition}
\begin{proof}
Claim (1) is verified by straightforward computation, which is left to the
reader. Part (2) is obtained from (1) by replacing the
\bpm $\mathcal{S}$ with $\mathcal{S}^0$ in Remark \ref{rem:op},
and replacing the $R$-module $(X,x)$ with the $R^{\mathfrak{r}}$-module
$(Y,y^0)$.
\end{proof}

\begin{example}
Consider the \lb monad morphism coming from an algebra
homomorphism $\boldsymbol{\varphi}: {\bf R}\to {\bf T}$ in a braided monoidal
category, as in
Example \ref{ex:alg_sep}. For a left ${\bf R}$-module $(X,\mu)$,
the inner actions on ${\bf T}\otimes X$ in Proposition
\ref{prop:i-o_mod} (1) come out as
\begin{align*}
\mu_{i}^l &:=(\mathbf{T}\otimes \mu)\circ \alpha
_{\mathbf{T},\mathbf{R},X}\circ
(\chi _{\mathbf{R},\mathbf{T}}\otimes X)\circ \alpha _{\mathbf{R},\mathbf{T}
,X}^{-1} \\
\mu_{i}^r &:=\left[ \mathbf{t}\circ (\mathbf{T}\otimes \boldsymbol{\varphi}
)\otimes X\right] \circ \alpha _{\mathbf{T,R},X}^{-1}\circ (\mathbf{T}
\otimes \chi _{\mathbf{R},\mathbf{X}}^{-1})\circ \alpha
_{\mathbf{T,X,R}}.
\end{align*}
For a right ${\bf R}$-module $(Y,\nu)$, the outer actions on ${\bf
T}\otimes Y$ in Proposition \ref{prop:i-o_mod} (2) come out as
\begin{align*}
\nu _{o}^{l} &:=\left[ \mathbf{t}\circ (\boldsymbol{\varphi}\otimes \mathbf{
T})\otimes Y\right] \circ \alpha _{\mathbf{R},\mathbf{T},Y}^{-1},
\qquad
\nu _{o}^{r} =\left( \mathbf{T}\otimes \nu \right) \circ \alpha _{\mathbf{
T,Y,R}}.
\end{align*}
\end{example}

Composition of the forgetful functor ${}_{R}\mathcal{C}_R \to
{}_{R}\mathcal{C}$ with the functor in Proposition \ref{prop:i-o_mod} (1)
yields a functor $T_i: {}_{R}\mathcal{C}_R \rightarrow
{}_{R}\mathcal{C}_{R}$.
Symmetrically, composition of the forgetful functor ${}_{R}\mathcal{C}_R \to
{}_{R^{\mathfrak{r}}} \mathcal{C}$ with the functor in Proposition
  \ref{prop:i-o_mod} (2)
yields a functor $T_o: {}_{R}\mathcal{C}_R \rightarrow {}_{R}\mathcal{C}_{R}$.
In the following proposition some natural transformations between various
  composites of these endofunctors on ${}_{R}\mathcal{C}_R$, and the functors
  ${\widetilde R}$ and ${\widetilde R}^0$ in Proposition
  \ref{prop:RX-mod}, are studied.
\begin{proposition}\label{prop:T_i,T_o}
Let $\mathcal{S}:=(\mathcal{C},T,R,\mathfrak{l},\mathfrak{t},
\mathfrak{r})$ be a
\bpm and $\varphi:R \to T$ be a
\lb monad morphism. Consider the above endofunctors
${\widetilde R}$, ${\widetilde R}^0$, $T_i$ and $T_o$ on ${}_{R}\mathcal{C}_R$.
Then the following hold.
\begin{enumerate}
\item The mappings $\mathrm{Ob}({}_{R}\mathcal{C}_R) \to
    \mathrm{Mor}({}_{R}\mathcal{C}_R)$, $(X,x,x^0)\mapsto x_i$ and
    $(X,x,x^0)\mapsto
    x_i^0 $, defined in terms of the inner actions in Proposition
    \ref{prop:i-o_mod} (2),
    determine natural transformations ${\widetilde R} T_o\to T_o$.
\item The mappings $\mathrm{Ob}({}_{R}\mathcal{C}_R) \to
    \mathrm{Mor}({}_{R}\mathcal{C}_R)$, $(X,x,x^0)\mapsto x_o$ and
    $(X,x,x^0)\mapsto
    x_o^0$, defined in terms of the outer actions in Proposition
    \ref{prop:i-o_mod} (3),
    determine natural transformations ${\widetilde R}^0 T_i\to T_i$.
\item The $BD$-law $\mathfrak{t}$ defines a natural transformation
$T_{i}T_{o}\to T_{o}T_{i}$.
\item The $BD$-law $\mathfrak{l}$ defines a natural transformations
${\widetilde R} T_{o}\to T_{o}{\widetilde R}$ and
${\widetilde R}^0 T_{o}\to T_{o}{\widetilde R}^0$.
\end{enumerate}
\end{proposition}
\begin{proof} Note that part (2) is obtained from part (1)
by replacing the
\bpm
$\mathcal{S}$ with $\mathcal{S}^0$ in
Remark \ref{rem:op}, and using the isomorphism ${}_{R}\mathcal{C}_R\cong
{}_{R^{\mathfrak{r}}}\mathcal{C}_{R^{\mathfrak{r}}}$.

Claims (1), (3) and (4) are proven by straightforward but somewhat
lengthy computations. We illustrate the main steps on the example
of part (3).
We use the graphical representation of morphisms in a category again.
We denote the morphisms defining the $(R,R)$-bimodule structure of
$ T_{i}T_{o}X$ by $x_{io}$ and $x_{io}^{0}$, respectively.
Similarly, for the actions of $R$ and $R^{\mathfrak{r}}$ on
$T_{o}T_{i}X$ we use the notation $x_{oi}$ and $ x_{oi}^{0}$. The
diagrammatic representation of these morphisms is given in the
first frame of Figure~\ref{fig:tX_morphism}. In the second frame
we prove that $\mathfrak{t}X$ is a morphism of $R$-modules. Note
that for the first equality we used that $\mathfrak{l}$ is a
distributive law and $ \varphi $ satisfies (A~4). For the
graphical proof of the fact $\mathfrak{t} X $ is a morphism of
$R^{\mathfrak{r}}$-modules see the third frame of Figure~\ref
{fig:tX_morphism}. Note that, besides the properties of
$\mathfrak{l}$ and $ \varphi$ that we already used, we also need
condition (A~2) in Definition~\ref{de:septuple}.
\begin{figure}[h]
\begin{center}
\includegraphics[scale=.78]{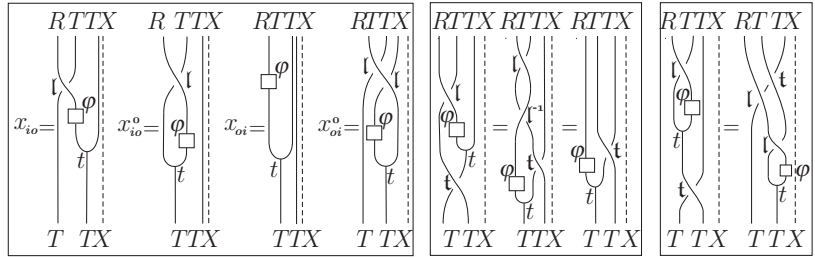}
\end{center}
\caption{$\mathfrak{t}X$ is a morphism of $(R,R)$-bimodules.}
\label{fig:tX_morphism}
\end{figure}
\end{proof}

\begin{proposition}\label{pr:T bar}
For a
\bpm $\mathcal{S}:=(\mathcal{C},T,R, \mathfrak{l},\mathfrak{t},
\mathfrak{r})$ and a
\lb monad morphism $\varphi:R \to T$, consider an $(R,R)$-bimodule
$(X,x,x^{0})$ and a $(T,R)$-bimodule $(Y,y,y^{0})$. Then the following
assertions hold.
\begin{enumerate}
\item The actions $x^{0}$ and $y^{0}$ induce $R^{\mathfrak{r}}$-module
  structures respectively on $\varphi _{\ast }(X,x)$ and $\varphi ^{\ast
  }(Y,y)$.

\item With respect to the above actions, $\varphi _{\ast }$ in Proposition
\ref{pr:adjoint functors} and $\varphi ^{\ast }$ in Proposition \ref{pr:phi_*}
can be regarded as functors $\varphi _{\ast }:{}_{R}\mathcal{C}
_{R}\rightarrow {}_{T}\mathcal{C}_R$ and $\varphi ^{\ast }:{}_{T}
\mathcal{C}_R \rightarrow {}_{R}\mathcal{C}_{R}$.

\item The functors constructed in part (2) define a pair of adjoint functors
$(\varphi _{\ast },\varphi ^{\ast })$.
\end{enumerate}
\end{proposition}

\begin{proof}
(1) and \ (2). Obviously, $y^{0}$ induces an $R^{\mathfrak{r}}$-action on
$\varphi^{\ast }(Y,y)=Y$. By construction, $\varphi ^{\ast }(Y,y)=(Y,y\circ
\varphi Y). $ We claim that $(Y,y\circ \varphi Y,y^{0})$ is an
$(R,R)$-bimodule. We have
to check that the module structures commute. Since $\varphi $ is a natural
morphism, condition (A~3) holds true and $y$ and $y^{0}$ commute, we get
\begin{equation*}
y\circ \varphi Y\circ Ry^{0}\circ \mathfrak{r}Y=y\circ Ty^{0}\circ
\varphi R^{0}Y\circ \mathfrak{r}Y=y\circ Ty^{0}\circ
\mathfrak{l}Y\circ R^{0}\varphi Y=y^{0}\circ R^{0}y\circ R^{0}\varphi
Y.
\end{equation*}
Hence $y\circ \varphi Y$ and $y^{0}$ commute too. Consequently, $\varphi
^{\ast }$ can be seen as a functor from $_{T}\mathcal{C}_R$ to ${}_{R}
\mathcal{C}_{R}.$

Recall from the proof of Proposition \ref{pr:adjoint functors} that the
$T$-module $(\varphi_\ast X,\overline{x})$ is the coequalizer of $(Tx,tX\circ
T\varphi X)$.
In terms of the inner actions in (\ref{ec: T_i}), the $T$-module morphisms
$Tx$ and $tX\circ T\varphi X$ are equal to $x_i\circ \mathfrak{l}^{-1} X$ and
$x_i^0\circ \mathfrak{l}^{-1} X$, respectively. Hence they are
$R^{\mathfrak{r}}$-module morphisms $T_o {\widetilde R} X\to T_o X$ by
Proposition \ref{prop:T_i,T_o} (1) and (4).
Thus $Tx$ and $tX\circ T\varphi X$ are morphisms of $(T,R)$-bimodules
from $(TRX,tRX,TRx^{0}\circ T\mathfrak{r}X\circ \mathfrak{l}RX)$ to
$(TX,tX,x_{o}^{0}).$ Hence, by Proposition \ref{le:coequalizer}, there is an
action $\overline{x}^0:R^{\mathfrak{r}}\varphi _{\ast }X\rightarrow \varphi
_{\ast }X$ such that
$(\varphi_\ast X,\overline{x},\overline{x}^{0})$ is a $(T,R)$-bimodule and
$\pi X$ is a morphism of $(T,R)$-bimodules form $(TX,tX,x_{o}^{0})$ to
$(\varphi_\ast X,\overline{x},\overline{x}^{0})$.

Let $f:(X,x,x^{0})\rightarrow (Z,z,z^{0})$ be a morphism of $(R,R)$-bimodules.
We already know that $\varphi _{\ast }f$ is a morphism of $T$-modules.
By Proposition \ref{prop:i-o_mod} (2), $T_of$ is an $R^{\mathfrak{r}}$-module
morphism. Hence so is $\varphi _{\ast }f$ by Proposition \ref{le:coequalizer}.
Thus $\varphi_{\ast }$ can be regarded as a functor from $_{R}\mathcal{C}_{R}$
to $_{T}\mathcal{C}_R$.

(3). Let $(X,x,x^{0})$ be an $(R,R)$-bimodule and let $
(Y,y,y^{0})$ be a $(T,R)$-bimodule. It is sufficient to show that
$\Phi_{X,Y}(f)$ and $\Theta _{X,Y}(g)$ are morphisms of
$R^{\mathfrak{r}}$-modules, for any $f:\varphi _{\ast }X\rightarrow Y$ in
$_{T}\mathcal{C}_R$ and any $
g:X\rightarrow \varphi ^{\ast }Y$ in $_{R}\mathcal{C}_{R}$ (for the
definition of $\Phi _{X,Y}$ and $\Theta _{X,Y}$ see the proof of Proposition
\ref{pr:adjoint functors}). The fact that $\Phi _{X,Y}(f)$ is a morphism of $
R^{\mathfrak{r}}$-modules is proved in the following computation:
\begin{equation*}
y^{0}\circ Rf\circ R\pi X \circ R^{0}\tau X=f\circ \overline{x}
^{0}\circ R\pi X\circ R\tau X=f\circ \pi X\circ Tx^{0}\circ
\mathfrak{l}X
\circ R\tau X=f\circ \pi X\circ \tau X\circ x^{0}.
\end{equation*}
Since $R\pi X$ is an epimorphism,
the relation meaning that $\Theta _{X,Y}(g)$ is a morphism of
$R^{\mathfrak{r}}$-modules follows from the computation below:
\begin{eqnarray*}
y^{0}\circ R^{\mathfrak{r}}\Theta _{X,Y}(g)\circ R^{\mathfrak{r}}\pi X&=&
y^{0}\circ R^{\mathfrak{r}}y\circ R^{\mathfrak{r}}Tg\overset{(A)}{=}
y\circ Ty^{0}\circ \mathfrak{l}Y\circ R^{\mathfrak{r}}Tg\overset{(B)}{=}
y\circ Tg\circ Tx^{0}\circ \mathfrak{l}X \\
&=&\Theta _{X,Y}(g)\circ \pi X\circ Tx^{0}\circ \mathfrak{l}X\overset{(C)}{=}
\Theta_{X,Y} (g)\circ \overline{x}^{0}\circ R^{\mathfrak{r}}\pi X.
\end{eqnarray*}
To get the equalities (A), (B) and (C) above we used the commutation
relation between the module structures of $Y$ and, respectively, that $T_og$
and $\pi $ are morphisms of $R^{\mathfrak{r}}$-modules.
\end{proof}

\begin{corollary}
For a
\bpm $\mathcal{S}:=(\mathcal{C},T,R,\mathfrak{l},\mathfrak{t},\mathfrak{r})$
and a
\lb monad morphism $\varphi:R \to T$,
consider the adjoint functors $(\varphi_\ast ,\varphi^\ast)$ in Proposition
\ref{pr:T bar}.

\begin{enumerate}
\item The unit of the adjunction $\sigma :\mathrm{Id}_{{}_{R}\mathcal{C}
_{R}}\rightarrow \varphi ^{\ast }\varphi _{\ast }$ is given, for any
  $(R,R)$-bimodule $(X,x,x^{0})$, by
\begin{equation}
\sigma X:=\pi X\circ \tau X.  \label{ec:unit adjunction}
\end{equation}

\item The counit of the adjunction $\xi :\varphi _{\ast }\varphi ^{\ast
}\rightarrow \mathrm{Id}_{{}_{T}\mathcal{C}_R}$ satisfies the following
relation for any $(T,R)$-bimodule $(Y,y,y^{0})$.
\begin{equation}
\xi Y\circ \pi \varphi ^{\ast }(Y,y,y^{0})=y.  \label{ec:counit
adjunction}
\end{equation}
\end{enumerate}
\end{corollary}

\begin{proof}
By definition, $\sigma X:=\Phi _{X,\varphi _{\ast }(X,x,x^{0})}(\mathrm{Id}
_{\varphi _{\ast }(X,x,x^{0})})$ and $\xi Y=\Theta _{\varphi ^{\ast
  }(Y,y,y^{0}),Y}( \mathrm{Id}_{\varphi ^{\ast }(Y,y,y^{0})}).$
\end{proof}

\begin{example}
For the
\lb monad morphism,
coming from an algebra homomorphism
$\boldsymbol{\varphi}: {\bf R}\to {\bf T}$ in a braided monoidal category as
in Example \ref{ex:alg_sep},
${\bf T}$ is a ${\bf T}$-${\bf R}$ bimodule via the left regular ${\bf
  T}$-action and the right ${\bf R}$-action induced by
$\boldsymbol{\varphi}$. If (as in Example \ref{ex:phi*}) the functors $X
\otimes (-)$ and $(-)\otimes X$ are right exact, for any object $X$ in
${\mathcal C}$, then the pair of adjoint functors in Proposition \ref{pr:T
  bar} (3) consists of the `restriction of scalars' functor $\varphi^\ast:{\bf
T}$-$\mathrm{Mod}$-${\bf
  R} \to  {\bf R}$-$\mathrm{Mod}$-${\bf R}$ and the induction functor
$\varphi_\ast={\bf T}\otimes_{\bf R} (-): {\bf R}$-$\mathrm{Mod}$-${\bf
  R} \to  {\bf T}$-$\mathrm{Mod}$-${\bf R}$.
\end{example}

\section{The para-cocyclic object associated to a braid preserving
  homomorphism}\label{sec:main}

In this section we show, using the main result in
\cite{BohmStefan}, that to every braid preserving monad morphism there
corresponds an admissible septuple, and hence a certain para-cocyclic object.

Recall that, for every pair $(F,G)$ of adjoint functors, with
$F:\mathcal{C} \rightarrow \mathcal{D}$ and
$G:\mathcal{D}\rightarrow \mathcal{C},$ the triple $(GF,G\xi
F,\sigma )$ is a monad on $\mathcal{C}$, where $\xi $ and $ \sigma
$ are respectively the counit and the unit of the adjunction. For
details the reader is referred to \cite[Main Application 8.6.2, p.
280]{We}.
In particular, the pair of adjoint functors $(\varphi _{\ast },\varphi ^{\ast
})$, constructed in Proposition \ref{pr:T bar} (3), determines a monad
structure on $\varphi^\ast \varphi_\ast:
{}_{R}\mathcal{C}_{R}\to {}_{R}\mathcal{C}_{R}$ that will be denoted by
$(\overline{T}_o,\overline{t}_o,\overline{\tau}_o)$.

Since $\mathfrak{l}$ is invertible, the object in $\mathcal{C}$, underlying
the $(R,R)$-bimodule $\overline{T}_o(X,x,x^0)$, is the coequalizer of the
morphisms $x_i$ and $x_i^0$ in (\ref{ec: T_i}). Since $x_i$ and $x_i^0$
are $(R,R)$-bimodule morphisms by Proposition \ref{prop:T_i,T_o} (1), it
follows by Proposition \ref{le:coequalizer} that
\begin{equation*}
\overline{T}_o(X,x,x^0)=\mathrm{Coeq}(x_{i},x_{i}^{0})
\end{equation*}
in ${}_R\mathcal{C}_R$.
That is, $\varphi^\ast$ takes the canonical epimorphism $\pi
X:(TX,tX,x_o^0)\rightarrow \varphi _{\ast }(X,x,x^{0})$ to a canonical
$(R,R)$-bimodule epimorphism
${\pi}_o X: T_o (X,x,x^{0})\to \overline{T}_o (X,x,x^{0})$.
In other words, the respective actions $\overline{x}_o$ and
$\overline{x}_o^{0}$ of $R$ and $R^{\mathfrak{r}}$ on $\,\overline{T}_o X$ are
uniquely determined such that ${\pi}_o X:T_{o}X\rightarrow \overline{T}_o X$
is a morphism of $(R,R)$ -bimodules, in the sense that
\begin{equation}\label{ec:x_o bar}
\overline{x}_o \circ R{\pi}_o X={\pi}_o X\circ x_o\qquad\text{and}\qquad
\overline{x}_o^{0}\circ R^{\mathfrak{r}}{\pi}_o X={\pi}_o X\circ x_{o}^{0}.
\end{equation}
The unit of $\overline{T}_o$ is
\begin{equation}\label{ec:tau bar}
\overline{\tau }_o={\pi}_o \circ \tau .
\end{equation}
In order to compute the multiplication $\overline{t}_o$ of $\overline{T}_o$,
substitute $Y=\varphi _{\ast}(X,x,x^{0})$ in (\ref{ec:counit adjunction}),
apply $\varphi^\ast$ on it, and compose the resulting relation on the right by
$T\pi_o X.$ Since $T{\pi}_o X$ is a coequalizer and the module structure
$\overline{x}$ satisfies (\ref{ec:x bar}), we deduce that $\overline{t}_oX$ is
the unique morphism in ${}_R\mathcal{C}_R$ satisfying relation
\begin{equation}
\overline{t}_o X\circ {\pi}_o \overline{T}_o X\circ T{\pi}_o X=
{\pi}_o X\circ tX.  \label{ec:t bar}
\end{equation}

Summarizing, we prove the following proposition.

\begin{proposition}\label{pr:T_o bar}
Associated to a
\bpm $\mathcal{S}:=(\mathcal{C},T,R,\mathfrak{l},\mathfrak{t}, \mathfrak{r})$
and a
\lb monad morphism $\varphi:R \to T$,
there is a monad
$(\overline{T}_o,\overline{t}_o,\overline{\tau}_o)$ on ${}_{R}\mathcal{C}_{R}$
such that, for every $(R,R)$-bimodule $(X,x,x^{0})$,
$$
(\overline{T}_o X,{\pi}_o X)=\mathrm{Coeq}(x_{i},x_{i}^{0}),
$$
(cf. Proposition \ref{prop:i-o_mod} (1)). That is, the $R$, and
$R^{\mathfrak{r}}$-actions $\overline{x}_o$ and $\overline{x}_o^{0}$ on
$\overline{T}_o X$ are $\ $uniquely defined such that ${\pi}_o
X:T_{o}X\rightarrow \overline{T}_o X$ is a morphism of $(R,R)$-bimodules.
The multiplication $\overline{t}_o:\overline{T}_o\overline{T}_o\rightarrow
$ $\overline{T}_o$ satisfies (\ref{ec:t bar}) and the unit $\overline{\tau
}_o: {}_R {\mathcal C}_R \to \overline{T}_o$ satisfies (\ref{ec:tau bar}).
\end{proposition}

A symmetrical version of Proposition \ref{pr:T_o bar} is obtained by replacing
the
\bpm $\mathcal{S}$ with $\mathcal{S}^0$, introduced in
Remark \ref{rem:op}.
\begin{proposition}\label{pr:T_i bar}
Associated to a
\bpm $\mathcal{S}:=(\mathcal{C},T,R,\mathfrak{l},\mathfrak{t},\mathfrak{r})$
and a
\lb monad morphism $\varphi:R \to T$,
there is a monad $(\overline{T}_i,\overline{t}_i,\overline{\tau }_i)$ on
${}_{R}\mathcal{C}_{R}$ such that, for every $(R,R)$-bimodule $(X,x,x^{0}),$
\begin{equation}
(\overline{T}_{i}X,{\pi}_{i}X)=\mathrm{Coeq}(x_{o},x_{o}^{0}),
\label{ec:T tilde}
\end{equation}
(cf. Proposition \ref{prop:i-o_mod} (2)). That is,
the actions $\overline{x}_i$ and $\overline{x}_i^{0}$ on
$\overline{T}_{i}X$ are
uniquely defined such that ${\pi}_{i}X:T_{i}X\rightarrow \overline{T}_{i}
X $ is a morphism of $(R,R)$-bimodules, i.e.
\begin{equation}
\overline{x}_i\circ R{\pi}_{i}X={\pi}_{i}X\circ x_{i}\text{\qquad and\qquad }
\overline{x}_i^{0}\circ R^{\mathfrak{r}}{\pi}_{i}X={\pi}_{i}X\circ x_{i}^{0}.
\label{ec:x tilde}
\end{equation}
The unit  $\overline{\tau }_i$
and the multiplication $\overline{t}_i$ satisfy
\begin{equation}\label{ec:t_i bar}
\overline{\tau }_i=\pi_i  \circ \tau
\qquad\textrm{and}\qquad
\overline{t}_i \circ {\pi}_i \overline{T}_i \circ T{\pi}_i =
{\pi}_i \circ t\circ \mathfrak{t}^{-1}.
\end{equation}
\end{proposition}

\begin{example}\label{ex:bar_T}
Let $\mathcal{C}$ be a braided monoidal category. For the \lb
monad morphism coming from an algebra homomorphism
$\boldsymbol{\varphi}: {\bf R}\to {\bf T}$ as in Example
\ref{ex:alg_sep}, ${\bf T}$ has a natural ${\bf R}$-bimodule
structure via $\boldsymbol{\varphi}$.
If (as in Example \ref{ex:phi*}) the functors $X
\otimes (-)$ and $(-)\otimes X$ are right exact, for any object $X$ in
${\mathcal C}$, then the functor underlying the
monad in Proposition \ref{pr:T_o bar} is induced by the ${\bf
R}$-bimodule ${\bf T}$, i.e. $\overline{T}_o={\bf T}\otimes_{\bf
R} (-):{\bf R}$-$\mathrm{Mod}$-${\bf R}
\to {\bf R}$-$\mathrm{Mod}$-${\bf R}$. Multiplication is given by
$({\overline{\bf t}}_o \otimes_{\bf R} (-))\circ {\overline{\alpha}}_{{\bf
    T},{\bf T}, (-)}^{-1}$, where $\overline{\alpha}$ is the associator
isomorphism
in ${\bf R}$-$\mathrm{Mod}$-${\bf R}$ and the ${\bf R}$-bimodule morphism
${\overline{\bf t}}_o:{\bf T}\otimes_{\bf R} {\bf T}\to {\bf T}$ is defined as
the projection of ${{\bf t}}:{\bf T}\otimes {\bf T}\to {\bf T}$. The
unit
is $\boldsymbol{\varphi} \otimes_{\bf R} (-)$ (where we used that ${\bf
  R}\otimes_{\bf R} (-)$ is naturally isomorphic to the identity functor on
${\bf R}$-$\mathrm{Mod}$-${\bf R}$).

Symmetrically, on an $\bf{R}$-bimodule $(X,\mu^l_X,\mu^r_X)$, the functor
$\overline{T}_i$ is defined as the coequalizer of $x_o=({\bf t}\otimes X)
\circ [(\boldsymbol{\varphi}\otimes {\bf T})\otimes X]\circ \alpha^{-1}_{{\bf
R},{\bf T},X}$ and
$x_o^0=(T\otimes \mu^r_X)\circ \alpha_{{\bf T},X,{\bf R}}\circ
\chi_{{\bf R},{\bf T}\otimes X}$. Equivalently, composing both $x_o$ and
$x_o^0$ by the isomorphism $\alpha_{{\bf R},{\bf T},X}\circ (\chi^{-1}_{{\bf
    R},{\bf T}} \otimes X)$, as a coequalizer
$$
\xymatrix{
({\bf T} \otimes {\bf R}) \otimes X \ar@<1mm>[rrrr]^-{[{\bf T} \otimes
      (\mu^r_X\circ \chi_{{\bf R},X})]\circ \alpha_{{\bf T},{\bf R},X}}
\ar@<-1mm>[rrrr]_-{[{\bf t}\circ \chi_{{\bf T},{\bf T}} \circ ({\bf T}
    \otimes  \boldsymbol{\varphi})]\otimes X} &&&&
{\bf T}\otimes X \ar[rr]^{\pi_i } && \overline{T}_i X.
}
$$
In order to obtain the form of the morphism corresponding to the lower one of
the parallel arrows, we used the first identity in \eqref{eq:br_ex}.
Note that  $\mu^r_X\circ \chi_{{\bf R},X}$ is a left action,
and ${\bf t}\circ \chi_{{\bf T},{\bf T}} \circ ({\bf T} \otimes
\boldsymbol{\varphi})$ is a right action for the algebra
${\bf R}^\chi:=({\bf R}, \boldsymbol{r}\circ \chi_{{\bf R},{\bf R}},
\boldsymbol{\rho})$.
Therefore, ${\overline{T}}_i X = {\bf T} \otimes_{{\bf R}^{\chi}} X$.
Since $\boldsymbol{\varphi}$ can be regarded as an algebra homomorphism ${\bf
  R}^\chi \to {\bf T}^{\chi^{-1}}$, the unit of the monad ${\overline{T}}_i (-) =
{\bf T} \otimes_{{\bf R}^{\chi}} (-)$ is given by $\boldsymbol{\varphi}
\otimes_{{\bf R}^{\chi}} (-)$. Multiplication is induced by the projection
${\overline{\bf t}}_i:{\bf T}\otimes_{{\bf R}^\chi} {\bf T}\to {\bf T}$ of the
${\bf R}^\chi$-bimodule morphism 
${{\bf t}}\circ \chi_{{\bf T},{\bf T}}^{-1}:{\bf T}\otimes {\bf T}\to {\bf T}$.
Note that, if $\chi$ is a symmetry, then ${\bf R}^\chi={\bf R}^{op}$, hence
${\overline T}_i = {\bf T}^{op} \otimes_{{\bf R}^{op}} (-)\cong (-)
\otimes_{\bf  R} {\bf T}$.

\end{example}

Next we introduce a functor $\Pi:{}_R\mathcal{C}_R\to \mathcal{C}$, that will
be used to construct a para-cocyclic module associated to a
\lb monad morphism.

\begin{definition}\label{def:Pi}
For a
\emph{BD}-law $\mathfrak{r}:RR\to RR$,
the functor $\Pi :{}_{R}\mathcal{C}_{R}\rightarrow \mathcal{C}$ is defined,
for an $(R,R)$-bimodule $(X,x,x^{0}),\ $by
\begin{equation*}
(\Pi X,pX):=\mathrm{Coeq}(x,x^{0}).
\end{equation*}
For every morphism $f:X\rightarrow Y$ of $(R,R)$-bimodules, $\Pi f$ is the
unique morphism in $\mathcal{C}$ such that $pY\circ f=\Pi f\circ pX.$ Hence
$p$ can be interpreted as a natural epimorphism from the forgetful functor
$U:{}_R \mathcal{C}_R\to \mathcal{C}$ to $\Pi$.
\end{definition}

In the graphical notation we do not denote the forgetful functor $U: {}_R
\mathcal{C}_R\to \mathcal{C}$. That is,
a box representing the natural transformation $p$ in Definition \ref{def:Pi}
has only a lower leg, corresponding to the functor $\Pi$.

\begin{remark}
In the context of Definition \ref{def:Pi}, the functor $\Pi T_i :{}_R
{\mathcal C}_R \to {\mathcal C}$ is equal to the composite of
${\overline T}_o:{}_R {\mathcal C}_R \to {}_R {\mathcal C}_R$ and the
forgetful functor $U:{}_R {\mathcal C}_R\to {\mathcal C}$.
Symmetrically, $\Pi T_o=U{\overline T}_i$.
\end{remark}

\begin{example}\label{ex:bimod_Pi}
For the \lb monad morphism coming from an algebra homomorphism
$\boldsymbol{\varphi}: {\bf R}\to {\bf T}$ as in Example
\ref{ex:alg_sep}, the functor $\Pi$ maps an ${\bf R}$-bimodule
$(X,\mu_X^l,\mu_X^r)$ to
\[
\Pi X= \textrm{Coker}(\mu_X^l-\mu_X^r\circ\chi_{{\bf R},X}).
\]
For two $\bf R$-bimodules $X$ and $Y$ we shall use the notation
$X\widehat{\otimes}_{\bf R}Y:= \Pi (X\otimes_{\bf R} Y)$ and we
shall say that this object in $\mathcal{C}$ is the \emph{braided
cyclic tensor product} of  $X$ and $Y$.
\end{example}

In the next theorem we prove that the \emph{BD}-law $\mathfrak{t}$, in a
\bpm $({\mathcal C}, T,R, \mathfrak{l}, \mathfrak{t}, \mathfrak{r})$, lifts to
a distributive law of the monads $\overline{T}_o$ and $\overline{T}_i$ in
Propositions \ref{pr:T_o bar} and \ref{pr:T_i bar}.

\begin{theorem}\label{te:t' and t bar}
Let $\mathcal{S}=(\mathcal{C},T,R,\mathfrak{l},\mathfrak{t}, \mathfrak{r}) $
be a
\bpm and
$\varphi:R\to T$ be a
\lb monad morphism.
Consider the associated functors $T_i$, $T_o$ and $\overline{T}_i$,
$\overline{T}_o$, and the natural epimorphisms $\pi_i:T_i\to \overline{T}_i$
and $\pi_o:T_o\to \overline{T}_o$ in Propositions \ref{pr:T_o bar} and
\ref{pr:T_i bar}.

\begin{enumerate}
\item There is a natural transformation $\mathfrak{t}^{\prime }:T_{i}
\overline{T}_o\rightarrow \overline{T}_oT_{i}$, between endofunctors on the
category of $(R,R)$-bimodules, such that
\begin{equation}  \label{ec:t'_1}
\mathfrak{t}^{\prime }\circ T_i{\pi_o}={\pi_o}T_i\circ \mathfrak{
t}.
\end{equation}

\item There is a natural transformation $\overline{\mathfrak{t}}:
\overline{T}_{i}\overline{T}_o \rightarrow \overline{T}_o\overline{T}_i $,
between endofunctors on the category of $(R,R)$-bimodules, such that
\begin{equation}
{\pi_o}\overline{T}_i \circ T_o{\pi}_i \circ \mathfrak{t}=
\overline{\mathfrak{t}}\circ {\pi}_i \overline{T}_o\circ
T_i\pi_o.  \label{ec:t induced}
\end{equation}
\item $\overline{\mathfrak{t}}:
\overline{T}_{i}\overline{T}_o \rightarrow \overline{T}_o\overline{T}_i $ is a
distributive law.
\end{enumerate}
\end{theorem}

\begin{proof}
(1). Let $(X,x,x^{0})$ be a given $(R,R)$-bimodule.  Since
$T$ preserves coequalizers, the sequence in the top row of the diagram in
  Figure~\ref {fig:t_bar} is a coequalizer.
\begin{figure}[h]
\begin{center}
\fbox{
\xymatrix{
T_i T_o{\widetilde R} X \ar@<1mm>[rrrr]^{f_1:=T_itX\circ T_i T_o \varphi X}
           \ar@<-1mm>[rrrr]_{g_1:=T_i T_o x}
           \ar[d]_{\mathfrak{t} {\widetilde R} X}&&&&
T_i T_o X \ar[r]^{T_i \pi_o X}
          \ar[d]_{\mathfrak{t}X}&
T_i\overline{T}_o X
          \ar@{.>}[d]^{\mathfrak{t}'X}\\
T_o T_i {\widetilde R} X \ar@<1mm>[rrrr]^{{f_2:=tT_iX\circ T_o\varphi T_iX\circ
                           T_o\mathfrak{l}^{-1}X}}
           \ar@<-1mm>[rrrr]_{g_2:=T_o T_i x}&&&&
T_o T_i X \ar[r]^{\pi_o T_i X}&
\overline{T}_o T_i X
}}
\end{center}
\caption{Construction of $\mathfrak{t}^{\prime}$.}
\label{fig:t_bar}
\end{figure}
Since $\mathfrak{l}$ is invertible, the sequence in the bottom row is a
coequalizer by the definition of $\pi_o$.
By Proposition \ref{prop:T_i,T_o} (3) the square whose horizontal edges are
$g_1$ and
$g_2$ is commutative. Using that $\mathfrak{t}$ is a \emph{BD}-law and taking
into account condition (A~4), it results that the square with horizontal
arrows $f_1$ and $f_2$ is also commutative.
Hence ${\pi_o}T_iX\circ \mathfrak{t}X$ coequalizes $(f_1,g_1)$.
Furthermore, by Proposition \ref{prop:T_i,T_o} (3),
$\mathfrak{t}X :T_{i}T_{o}X\rightarrow T_{o}T_{i}X$ is a morphism
of $(R,R)$-bimodules. By
construction of $\overline{T}_o,$ the canonical epimorphism ${\pi_o}X
:T_{o}X\rightarrow \overline{T}_oX$ is a morphisms of $(R,R)$-bimodules.
Hence ${\pi_o}T_{i}X:T_{o}T_{i}X\rightarrow \overline{T}_oT_{i}X$ and
$T_{i}{\pi_o}X:T_{i}T_{o}X\rightarrow T_{i}\overline{T}_oX$ are also
morphisms in $_{R}\mathcal{C}_{R}.$ Moreover,
$tX\circ T\varphi X$ and $Tx$ are $(R,R)$-bimodule morphisms by
Proposition \ref{prop:T_i,T_o} (1) and (4). Hence so are $f_1$ and $g_1$ by
Proposition \ref{prop:i-o_mod} (1).
In conclusion, we can apply Proposition \ref{le:coequalizer} in the category
$_R \mathcal{C}_R$ to show that there exists an $(R,R)$-bimodule morphism
$\mathfrak{t}^{\prime }X$, rendering commutative the right hand side square on
Figure~\ref {fig:t_bar}. By construction, $\mathfrak{t}^{\prime }$ is natural.

(2). It follows by condition (A~4), naturality of $\varphi$
   and the fact that $\mathfrak{t}$ is a distributive law that the morphisms
   $x_o$ and $x_o^0$ in (\ref{ec: T_o}) satisfy
\begin{equation}\label{eq:sq1}
T_o x_o\circ \mathfrak{l} T_i X\circ R \mathfrak{t} X = \mathfrak{t}X \circ
tT_o X \circ \varphi T_i T_o X\qquad \textrm{and}\qquad
T_i x_o^0 \circ \mathfrak{l} T_i X \circ R \mathfrak{t} X = \mathfrak{t}X
\circ T_i x_o^0 \circ \mathfrak{l} T_o X.
\end{equation}
Taking into account (\ref{ec:t'_1}), the identities in (\ref{eq:sq1}) imply
\begin{eqnarray}\label{eq:sq2}
&&\mathfrak{t}' X\circ T_i\pi_oX\circ tT_o X \circ \varphi T_i T_o X=
\pi_o T_i X \circ T_o x_o\circ \mathfrak{l} T_i X\circ R \mathfrak{t} X
\qquad \textrm{and}\nonumber\\
&&\mathfrak{t}' X\circ T_i\pi_oX\circ T_i x_o^0 \circ \mathfrak{l} T_o X =
\pi_o T_i X \circ T_o x_o^0 \circ \mathfrak{l} T_i X \circ R \mathfrak{t} X.
\end{eqnarray}
Using the naturality of $t$, $\varphi$ and $\pi_o$ together with the
$(R,R)$-bimodule morphism property (\ref{ec:x_o bar})
of $\pi_o$, the equations in (\ref{eq:sq2}) can be seen to
be equivalent to the commutativity of the left hand side squares on Figure
\ref{fig:bar t}.
\begin{figure}[h]
\begin{center}
\fbox{
\xymatrix{
{\widetilde R}^0T_i T_o X \ar[r]^{{\widetilde R}^0T_i \pi_o X}
\ar[d]_{{\widetilde R}^0\mathfrak{t}X}& {\widetilde R}^0T_i \overline{T}_o X
\ar@<.5mm>[rr]^{t\overline{T}_o X\circ \varphi T_i \overline{T}_o X}
\ar@<-.5mm>[rr]_{T_i \overline{x}_o^0\circ \mathfrak{l}\overline{T}_o X}&&
T_i \overline{T}_o X \ar[r]^{\pi_i \overline{T}_o X}
\ar[dd]^{\mathfrak{t}'X}&
\overline{T}_i \overline{T}_o X
\ar@{.>}[dd]^{\overline{\mathfrak{t}}X}\\
{\widetilde R}^0T_oT_iX\ar[d]_{\mathfrak{l} T_iX}& && &\\
T_o {\widetilde R}^0 T_i X \ar[r]^{\pi_o {\widetilde R}^0 T_i X}&
\overline{T}_o {\widetilde R}^0 T_i X
\ar@<.5mm>[rr]^{\overline{T}_o x_o}
\ar@<-.5mm>[rr]_{\overline{T}_o x_o^0}&&
\overline{T}_o T_i X \ar[r]^{\overline{T}_o \pi_i X}&
\overline{T}_o \overline{T}_i X
}}
\end{center}
\caption{Construction of $\overline{\mathfrak{t}}$.}
\label{fig:bar t}
\end{figure}
Since $\pi_i X$ coequalizes $(x_o,x_o^0)$, the morphism $\overline{T}_o \pi_i
X$ coequalizes $\overline{T}_o x_o \circ \pi_o {\widetilde R}^0 T_i X \circ
\mathfrak{l} T_i X \circ {\widetilde R}^0 \mathfrak{t} X$ and $\overline{T}_o
x_o^0 \circ \pi_o {\widetilde R}^0 T_i X \circ \mathfrak{l} T_i X \circ
{\widetilde R}^0\mathfrak{t} X$. Hence $\overline{T}_o \pi_i
X\circ \mathfrak{t}'X$ coequalizes
$t\overline{T}_o X\circ \varphi T_i \overline{T}_o X\circ
{\widetilde R}^0T_i \pi_o X$ and $T_i \overline{x}_o^0\circ
\mathfrak{l}\overline{T}_o X \circ {\widetilde R}^0T_i \pi_o X$. Since
${\widetilde R}^0T_i \pi_o X$ is an
epimorphism, this implies the existence of a morphism
$\overline{\mathfrak{t}}X$ in $\mathcal{C}$, rendering commutative the right
hand side square on Figure \ref{fig:bar t}.
Furthermore, the parallel arrows $t\overline{T}_o X\circ \varphi T_i
\overline{T}_o X$ and $T_i \overline{x}_o^0\circ \mathfrak{l}\overline{T}_o X$
on Figure \ref{fig:bar t} are morphisms of $(R,R)$-bimodules
by Proposition \ref{prop:T_i,T_o} (2) and $\mathfrak{t}'X$ is a morphism of
$(R,R)$-bimodules by part (1).
By applying Proposition \ref{le:coequalizer} in $_R\mathcal{C}_R$, we conclude
that $\overline{\mathfrak{t}}X$ is an $(R,R)$-bimodule morphism.
Naturality of $\overline{\mathfrak{t}}$ follows obviously
from the naturality of $\mathfrak{t}'$.

(3). The first axiom in Definition \ref{de:distributive law} is
proven in the left frame in Figure~\ref{fig:tbar}. The first and
fourth equalities follow by \eqref{ec:t induced} and the fact that
$\pi_i$ and $\pi_o$ are natural transformations. The second and fifth
equalities are consequences of the second identity in \eqref{ec:t_i
bar}. For the third relation we used \emph{YB}-equation on
$\mathfrak{t}$ and that $\mathfrak{t}$ is a distributive law.
Since $ \pi_i {\overline T}_i {\overline T}_o\circ T \pi_i
{\overline T}_o \circ TT \pi_o$ is epi, this proves the first
axiom in the definition of distributivity laws.
\begin{figure}[h]
\begin{center}
\includegraphics[scale=.78]{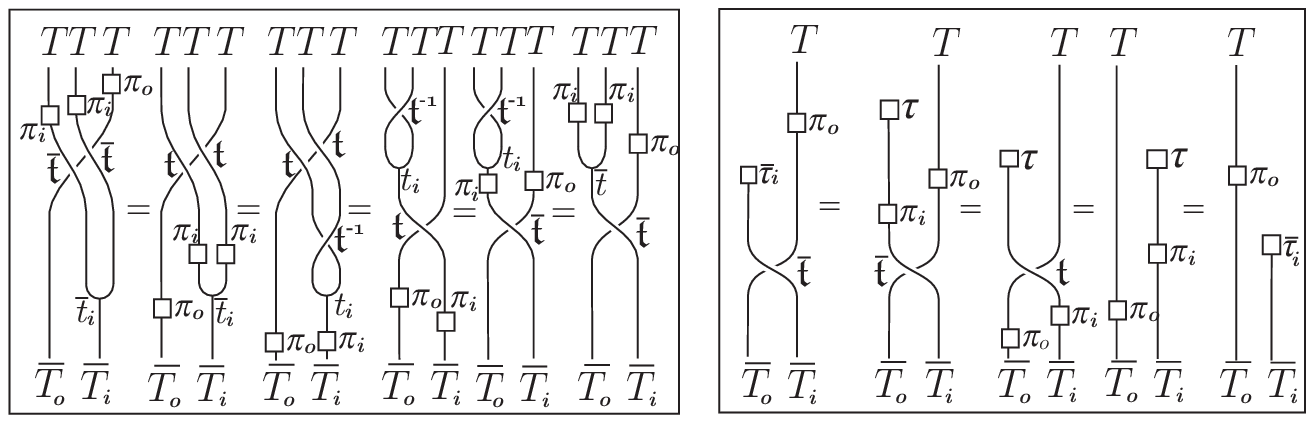}
\end{center}
\caption{$\bar{\mathfrak{t}}$ is a distributivity law.}
\label{fig:tbar}
\end{figure}

Similarly, taking into account the first identity in \eqref{ec:t_i
bar} we get the first and last equalities in the right frame of
Figure~\ref{fig:tbar}. By \eqref{ec:t induced} and the fact that
$\mathfrak{t}$ is a distributive law we deduce the second and the
third identities. Note that in the first, second and last
equalities we also used naturality which, in this case, means that e.g.
$\pi_o$ can by pushed up and down along a string. Since $\pi_o$ is
an epimorphism, this proves that $\overline{\mathfrak{t}}$
satisfies the third axiom of a distributive law in Definition
\ref{de:distributive law}. The remaining two axioms are verified
similarly.
\end{proof}

\begin{theorem}\label{thm:i}
Take a
\bpm $\mathcal{S}=(\mathcal{C},T,R,\mathfrak{l},\mathfrak{t}, \mathfrak{r})$
and a
\lb monad morphism $\varphi:R \to T$.
Consider the functors
$\overline{T}_o$ and $\overline{T}_i$, constructed in
Propositions \ref{pr:T_o bar} and \ref{pr:T_i bar}, respectively, and the
functor $\Pi$ introduced in Definition \ref{def:Pi}. For these data
there are mutually inverse natural isomorphisms $\mathfrak{i}:\Pi\,
\overline{T}_o \rightarrow \Pi\, \overline{T}_i $ and $\mathfrak{j}:\Pi\,
\overline{T}_i \rightarrow \Pi\, \overline{T}_o$, between functors ${}_R
\mathcal{C}_R \to \mathcal{C}$, such that
\begin{equation}\label{eq:i,j}
\mathfrak{i}\circ p\overline{T}_o\circ
{\pi}_o=p\overline{T}_i \circ {\pi}_i ,\qquad\qquad
\mathfrak{j}\circ p\overline{T}_i \circ
{\pi}_i =p\overline{T}_o\circ {\pi}_o.
\end{equation}
\end{theorem}

\begin{proof}
In the first diagram of Figure \ref{fig:morfism i} the squares are
commutative, for any $(R,R)$-bimodule $(X,x,x^0)$,
as ${\pi}_i X:T_{i}X\rightarrow \overline{T}_i X$ is a morphism of $(R,R)$
-bimodules. Hence $p\overline{T}_i X\circ {\pi}_i X$ coequalizes $
(x_{i},x_{i}^{0}).$ Then there is a morphism $\mathfrak{i}^{\prime }X:
\overline{T}_oX\rightarrow \Pi\, \overline{T}_i X$ such that
$\mathfrak{i}^{\prime }X\circ {\pi}_oX=p\overline{T}_i X\circ
{\pi}_i X.$
\begin{figure}[h]
\begin{center}
\fbox{
\begin{xy} (0,0)*+{}="v1";(20,0)*+{\Pi\, \overline{T}_i X}="v2";(40,0)*+{}="v3"; (0,15)*+{R\overline{T}_i
X}="v4";(20,15)*+{\overline{T}_i  X}="v5";(40,15)*+{}="v6";
(0,30)*+{RT_i X}="v7";(20,30)*+{T_iX}="v8";(40,30)*+{\overline{T}_o
X}="v9"; {\ar@<.5mm>^-{x_i} "v7";"v8" }; {\ar@<-1mm>_-{x_i^{0}}
"v7";"v8"}; {\ar@<1mm>^-{\overline{x}_i} "v4";"v5" };
{\ar@<-.5mm>_-{\overline{x}_i^{0}} "v4";"v5"}; {\ar@{->}_{R{\pi}_i
X} "v7"; "v4"}; {\ar@{->}^{{\pi}_i  X} "v8"; "v5"}; {\ar@{->}_{p
\overline{T}_i  X} "v5"; "v2"}; {\ar@{.>}^{\mathfrak{i}'X} "v9";
"v2"}; {\ar@{->}^{\pi_o X} "v8"; "v9"};
\end{xy}\qquad\qquad
\begin{xy} (0,0)*+{}="v1";(20,0)*+{\Pi\,
\overline{T}_i X}="v2";(40,0)*+{}="v3";
(0,15)*+{R\overline{T}_o X}="v4";(20,15)*+{\overline{T}_o
X}="v5";(40,15)*+{\Pi\,\overline{T}_oX}="v6";
(0,30)*+{RT_oX}="v7";(20,30)*+{T_oX}="v8";(40,30)*+{}="v9";
{\ar@<.5mm>^-{x_o} "v7";"v8" }; {\ar@<-1mm>_-{x_o^{0}} "v7";"v8"};
{\ar@<1mm>^-{\overline{x}_o} "v4";"v5" };
{\ar@<-.5mm>_-{\overline{x}_o^{0}} "v4";"v5"}; {\ar@{->}_{R \pi_o X}
"v7"; "v4"}; {\ar@{->}^{\pi_o X} "v8"; "v5"};
{\ar@{->}_{\mathfrak{i}'X} "v5"; "v2"}; {\ar@{.>}^{\mathfrak{i}X}
"v6"; "v2"}; {\ar@{->}^{p\overline{T}_o X} "v5"; "v6"};
\end{xy}}
\end{center}
\caption{The construction of $\mathfrak{i}$.}
\label{fig:morfism i}
\end{figure}
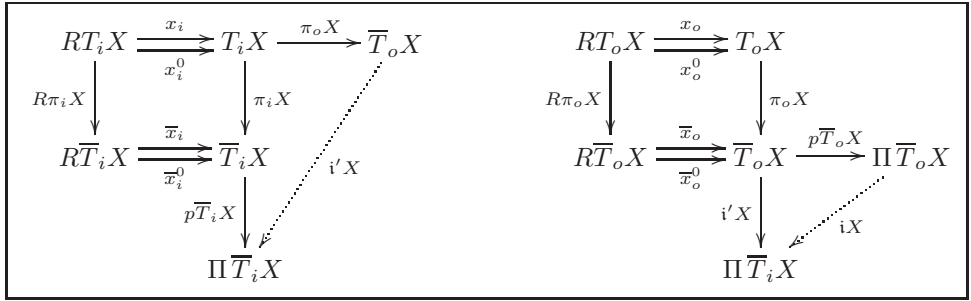
In the second diagram of Figure \ref{fig:morfism i}, the squares are also
commutative, as ${\pi}_oX$ is a morphism of $(R,R)$-bimodules. Therefore $
\mathfrak{i}^{\prime }X $ coequalizes $(\overline{x}_o,\overline{x}_o^{0}),$
as $ {\pi}_i X$ coequalizes $(x_{o},x_{o}^{0})$ and $R{\pi}_oX$ is an
epimorphism. Thus there exists a morphism $\mathfrak{i}X:\Pi\,
\overline{T}_o X\rightarrow \Pi\, \overline{T}_i X$ satisfying the
required relation. To construct $\mathfrak{j}$ one proceeds
analogously.
The natural morphisms $\mathfrak{i}$ and $\mathfrak{j}$ are
mutual inverses, as $p\overline{T}_i X\circ {\pi}_i X$ and
$p\overline{T}_oX\circ {\pi}_oX$ are epimorphisms.
Naturality of $\mathfrak{i}$ follows by (\ref{eq:i,j}) and naturality of $p$,
$\pi_o$ and $\pi_i$.
\end{proof}

\begin{example} \label{ex:bimod_i}
Let $\mathcal{C}$ be an abelian braided monoidal category such
that the functors $X\otimes (-)$ and $(-)\otimes X$ are right
exact, for any object $X$ in ${\mathcal C}$ (cf. Example
\ref{ex:phi*}). For the monad morphism coming from an algebra
homomorphism $\boldsymbol{\varphi}: {\bf R}\to {\bf T}$ as in
Example \ref{ex:alg_sep}, the natural transformation
$\overline{\mathfrak{t}}$ in Theorem \ref{te:t' and t bar} (2) is
the projection of $\alpha_{{\bf T}, {\bf T}, X}
\circ (\chi_{{\bf T},{\bf T}}\otimes X) \circ \alpha^{-1}_{{\bf T}, {\bf T},
  X}$. That is, it is the unique morphism satisfying
$$
\pi_o ({\bf T}\BP X) \circ ({\bf T}\otimes \pi_i X) \circ \alpha_{{\bf T}, {\bf
    T}, X} \circ (\chi_{{\bf T},{\bf T}}\otimes X) \circ \alpha^{-1}_{{\bf T},
  {\bf T}, X} =\overline{\mathfrak{t}}X \circ \pi_i ({\bf T}\otimes_{\bf R} X)
\circ (T\otimes \pi_o X),
$$
where notations in Example \ref{ex:bar_T} are used.

For an ${\bf R}$-bimodule $X$, the natural transformation
$\mathfrak{i}X: {\bf T}\widehat{\otimes}_{\bf R} X\to \Pi({\bf T} \BP X)$,
constructed in the previous theorem, is the projection of the identity map
${\bf T}\otimes X \to{\bf T}\otimes X$ .
\end{example}

As it is explained in \cite[page 281]{We}, for any monad $({\bf T},
  {\boldsymbol \mu},{\boldsymbol \eta})$ on a category ${\mathcal M}$ and an
  object $X$ in ${\mathcal M}$, there is an associated cosimplex in ${\mathcal
  M}$, given at degree $n$ by ${\bf T}^{n+1} X$. Coface and
  codegeneracy maps are given, for $k=0,\dots,n$, by
$$
{\bf T}^k {\boldsymbol\eta} {\bf T}^{n-k} X:{\bf T}^{n} X \to {\bf T}^{n+1} X
\qquad \textrm{and}\qquad
{\bf T}^k {\boldsymbol\mu} {\bf T}^{n-k} X:{\bf T}^{n+2} X \to {\bf T}^{n+1}
X,
$$
respectively. Clearly, application of any functor ${\boldsymbol \Pi}:{\mathcal
  M}\to {\mathcal C}$ yields a cosimplex in ${\mathcal C}$. In
\cite{BohmStefan}, para-cocyclic structures on the resulting cosimplex in
${\mathcal C}$ were studied. Recall from \cite{BohmStefan} the following
construction.

\begin{definition}\cite[Definitions 1.7 and 1.8]{BohmStefan}
\label{def:adm_sep}
An \emph{admissible septuple} ${\mathcal A}$ consists of the data $({\mathcal
  M}, {\mathcal C}, {\bf T}_o,$ ${\bf T}_i, {\boldsymbol
  \Pi},{\boldsymbol t}, {\boldsymbol i})$, where
\begin{itemize}
\item $\mathcal{C}$ and $\mathcal{M}$ are categories,
\item $({\bf T}_o, {\boldsymbol \mu}_o,{\boldsymbol \eta}_o)$ and
  $({\bf T}_i,{\boldsymbol \mu}_i,{\boldsymbol \eta}_i)$ are monads on
  $\mathcal{M}$,
\item ${\boldsymbol \Pi}$ is a functor $\mathcal{M}\to \mathcal{C}$,
\item ${\boldsymbol{t}}:{\bf T}_i {\bf T}_o \to {\bf T}_o {\bf T}_i$ is a
  distributive law,
\item  ${\boldsymbol i}:{\boldsymbol\Pi} {\bf T}_o\to {\boldsymbol\Pi}
  {\bf T}_i$ is a natural transformation,
\end{itemize}
subject to the conditions
\begin{equation}
{\boldsymbol i}\, \circ {\boldsymbol \Pi}\, {\boldsymbol\eta}_o\, =
{\boldsymbol \Pi}\, {\boldsymbol \eta_i}
\qquad \textrm{and}\qquad
{\boldsymbol  i}\, \circ {\boldsymbol\Pi}\, {\boldsymbol \mu}_o\, =
{\boldsymbol \Pi}\, {\boldsymbol \mu}_i \circ \, { \boldsymbol
  i}\,{\bf T}_i\, \ \circ {\boldsymbol\Pi}\,  {\boldsymbol t} \circ
{\boldsymbol i}\, {\bf T}_o . \label{eq:i}
\end{equation}
A \emph{transposition morphism} for the admissible septuple $\mathcal {A}$ is
a pair $(X,w)$, consisting of an object $X$ and a morphism $w:{\bf T}_i
X \to {\bf T}_o X$ in ${\mathcal M}$, satisfying
\begin{equation}
w\circ {\boldsymbol \eta}_i X ={\boldsymbol \eta}_o X
\qquad \textrm{and}\qquad
w\circ { \boldsymbol \mu}_i X = {\boldsymbol \mu}_o X \circ {\bf T}_o
w \circ {\boldsymbol t} X \circ {\bf T}_i w
\label{eq:w}.
\end{equation}
\end{definition}

\begin{theorem}\cite[Theorem 1.10]{BohmStefan}\label{thm:coc_distr_law}
Consider an admissible septuple $({\mathcal
  M}, {\mathcal C}, {\bf T}_o,{\bf T}_i, {\boldsymbol
  \Pi},{\boldsymbol t}, {\boldsymbol i})$ and a transposition morphism $(X,w)$
  for it.
The associated cosimplex $Z^\ast:={\boldsymbol \Pi} {\bf T}_o^{\ast+1}
X$ is para-cocyclic with para-cocyclic morphism
\begin{equation}\label{eq:para_coc_op_distr_law}
{w}_n:={\boldsymbol\Pi} {\bf{T}}_o^n w \circ {\boldsymbol\Pi}
{\bf{T}}_o^{n-1} \, {\boldsymbol t} \, X \circ {\boldsymbol\Pi}
{\bf{T}}_o^{n-2} \, {\boldsymbol t}\,
  {\bf{T}}_o X \circ \dots \circ {\boldsymbol\Pi} {\bf{T}}_o
  \, {\boldsymbol t}\,   {\bf{T}}_o^{n-2} X \circ
{\boldsymbol\Pi} \, {\boldsymbol t}\,   {\bf{T}}_o^{n-1}  X \circ
{\boldsymbol i}\,  {\bf{T}}_o^n X.
\end{equation}
\end{theorem}

The next theorem is our main result.

\begin{theorem} \label{thm:main}
Consider a
\bpm $\mathcal{S}=(\mathcal{C},T,R,\mathfrak{l}, \mathfrak{t}, \mathfrak{r})$
and a
\lb monad morphism $\varphi:R \to T$. Then the data, consisting of
\begin{itemize}
\item the categories $\mathcal{C}$ and $\mathcal{M}:={}_R \mathcal{C}_R$,
\item the monads $\overline{T}_o$ in Proposition \ref{pr:T_o bar} and
  $\overline{T}_i$ in Proposition \ref{pr:T_i bar},
\item the functor $\Pi: {}_R \mathcal{C}_R \to \mathcal{C}$ in Definition
  \ref{def:Pi},
\item the natural transformations $\overline{\mathfrak{t}}:\overline{T}_i\,
  \overline{T}_o \to \overline{T}_o \, \overline{T}_i$, constructed in Theorem
  \ref{te:t' and t bar} (2), and $\mathfrak{i}:\Pi \overline{T}_o \to \Pi
  \overline{T}_i$, constructed in Theorem \ref{thm:i},
\end{itemize}
constitute an admissible septuple in the sense of Definition
\ref{def:adm_sep}. Moreover, a
transposition morphism for it consists of an $(R,R)$-bimodule $X$ and an
$(R,R)$-bimodule map $w:\overline{T}_i X \to \overline{T}_o X$, satisfying
\begin{equation}
\widetilde{w}\circ \tau X =\overline{\tau }_o X, \qquad
\textrm{and}\qquad \widetilde{w}\circ tX =\overline{t}_oX\circ
\overline{T}_o\widetilde{w} \circ \mathfrak{t}^{\prime }X\circ
T\widetilde{w}\circ \mathfrak{t}X, \label{eq:w_ex}
\end{equation}
where $\widetilde{w}$ is defined in terms of the natural morphism $\pi_i$ in
Proposition \ref{pr:T_i bar} as $\widetilde{w}:=w\circ {\pi}_i X$ and the
natural transformation $\mathfrak{t}^{\prime }:T_i \overline{T}_o\rightarrow
\overline{T}_oT_i$ is defined in Theorem \ref{te:t' and t bar} (1).
\end{theorem}

\begin{proof}
$\overline{T}_o$ and $\overline{T}_i$ are monads on ${}_R {\mathcal C}_R$ by
 Proposition \ref{pr:T_o bar} and Proposition \ref{pr:T_i
   bar}, respectively. $\Pi$ is a functor ${}_R {\mathcal C}_R\to {\mathcal C}$
 by Definition \ref{def:Pi}. $\overline{\mathfrak t}$ is a distributive law
  $\overline{T}_i \overline{T}_o \to \overline{T}_o \overline{T}_i$ by Theorem
  \ref{te:t' and t bar} (3). Thus in order to verify that they constitute an
  admissible septuple, we have to prove that the natural transformation
  $\mathfrak{i}: \Pi {\overline T}_o \to \Pi {\overline T}_i$ in Theorem
  \ref{thm:i} satisfies
  conditions \eqref{eq:i}. In terms of the natural epimorphism $p$ from the
  forgetful functor ${}_R {\mathcal C}_R \to {\mathcal C}$ to $\Pi$,
\begin{equation}\label{eq:i_un}
{\mathfrak i} \circ \Pi {\overline \tau}_o \circ p =
{\mathfrak i} \circ \Pi \pi_o \circ \Pi \tau \circ p =
{\mathfrak i} \circ \Pi \pi_o \circ p T \circ \tau =
p {\overline T}_i \circ \pi_i \circ \tau =
p {\overline T}_i \circ {\overline \tau}_i =
\Pi {\overline \tau}_i \circ p.
\end{equation}
In the first equality we used \eqref{ec:tau bar} and in the
penultimate equality we used the first identity in \eqref{ec:t_i
bar}. The third equality is a consequence of the first identity in
\eqref{eq:i,j}. The other equalities follow by naturality. Since
$p$ is epi, \eqref{eq:i_un} proves the first condition in
\eqref{eq:i}.
\begin{figure}[h]
\begin{center}
\includegraphics[scale=.78]{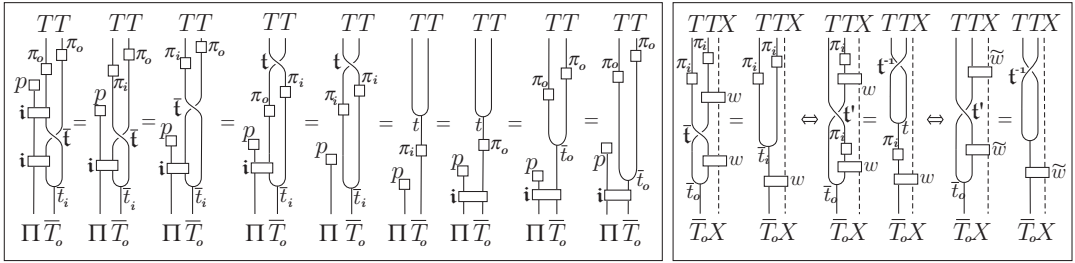}
\end{center}
\caption{The proof of Theorem \ref{thm:main} (3).} \label{fig:as1}
\end{figure}

The second condition in \eqref{eq:i} is proven in
Figure~\ref{fig:as1}. The first, fourth and seventh equalities
follow by the first identity in \eqref{eq:i,j}. The third equality
is a consequence of \eqref{ec:t
  induced}. In the fifth equality we used the second condition in
\eqref{ec:t_i bar} and in the seventh one we used \eqref{ec:t
bar}. The other equalities follow by naturality. Since $p
{\overline T}_o {\overline T}_o  \circ \pi_o {\overline T}_o\circ
T \pi_o$ is epi, we have the second condition in \eqref{eq:i}
proven.

The last claim about transposition morphisms is proven by showing
the equivalence of conditions \eqref{eq:w_ex} and \eqref{eq:w}. In
view of the first condition in \eqref{ec:t_i bar}, the first
conditions in \eqref{eq:w_ex} and \eqref{eq:w} are obviously
equivalent. In the second frame of Figure~\ref{fig:as1} it is
shown that the other conditions in \eqref{eq:w_ex} and
\eqref{eq:w} are equivalent. Composition on the right of both
sides of the second condition in \eqref{eq:w} with the epimorphism
$\pi_i {\overline T}_i X \circ T \pi_i X$ yields an equivalent
condition, namely the first equality in the above mentioned
picture. This is equivalent to the second equality, in view of
\eqref{ec:t_i bar} together with the fact that, by \eqref{ec:t'_1} and
\eqref{ec:t induced}, ${\overline T}_o \pi_i \circ {\mathfrak t}' = {\overline
  {\mathfrak t}}\circ \pi_i {\overline T}_o$. To conclude the proof
we use $\widetilde{w}=w\circ {\pi}_i X$.
\end{proof}

Combining Theorem \ref{thm:main} with Theorem \ref{thm:coc_distr_law}, we
obtain the following
\begin{corollary}\label{cor:P-Cocyc}
Let $\mathcal{S}=(\mathcal{C},T,R,\mathfrak{l}, \mathfrak{t}, \mathfrak{r})$
be a
\bpm
and $\varphi:R \to T$ be a
\lb monad morphism. Consider the monads $\overline{T}_o$ on ${}_R {\mathcal
C}_R$ in Proposition \ref{pr:T_o bar} and $\overline{T}_i$ in Proposition
\ref{pr:T_i bar} and the functor $\Pi: {}_R \mathcal{C}_R
\to \mathcal{C}$ in Definition \ref{def:Pi}. Then any $(R,R)$-bimodule morphism
$w:\overline{T}_i X \to \overline{T}_o X$, satisfying \eqref{eq:w_ex},
determines a para-cocyclic structure on the cosimplex $\Pi
\overline{T}_o^{\ast +1} X$. The para-cocyclic morphism is given in terms of
the natural transformation $\mathfrak{i}$, constructed in Theorem \ref{thm:i},
as
$$
\Pi  \overline{T}_o^{n} w \circ
\Pi\, \overline{T}_o^{n-1} \,
\overline{\mathfrak{t}}  \circ \Pi\, \overline{T}_o^{n-2}\,
\overline{\mathfrak{t}}\, \overline{T}_o^{{}}\circ \cdots \circ \Pi\,
\overline{T}_o\,  \overline{\mathfrak{t}}\, \overline{T}_o^{n-2}\circ
\Pi\, \overline{\mathfrak{t}}\, \overline{T}_o^{n-1} \circ
\mathfrak{i}\, \overline{T}_o^{n}: \Pi \overline{T}_o^{n +1} X \to \Pi
\overline{T}_o^{n +1} X.
$$
\end{corollary}

An immediate example of the situation in Corollary
\ref{cor:P-Cocyc} is induced by an algebra homomorphism in a braided monoidal
category. 

\begin{example}\label{ex:braid}
In this example, in order to simplify formulae, we do not write the
associativity constraints.
Let $\boldsymbol{\varphi}: {\bf R} \to {\bf T}$ be a morphism of
algebras in a braided monoidal category, in which the functors $X \otimes (-)$
and $(-)\otimes X$ are right exact, for any object $X$ in ${\mathcal C}$.
Assume that $\boldsymbol{\varphi}$ satisfies \eqref{eq:br_ex} and consider the
corresponding \lb monad morphism in Example 
\ref{ex:alg_sep}. It determines an admissible septuple
$(\mathcal{C}, {\bf R}$-$\mathrm{Mod}$-${\bf R}, {\bf T}
\otimes_{\bf R} (-), {\bf T}\BP (-), \Pi, \overline{\mathfrak{t}},
{\mathfrak i})$, where the functors ${\bf T}\otimes_{\bf R} (-)$, ${\bf T}\BP
(-)$ and $\Pi$ are described
in Example \ref{ex:bar_T} and Example \ref{ex:bimod_Pi}, respectively, and the
natural transformations ${\overline{\mathfrak{t}}}$ and ${\mathfrak i}$
can be found in Example \ref{ex:bimod_i}. For any ${\bf
R}$-bimodule $X$ there is a corresponding cosimplex, given at
degree $n$ by
 $$
\mathrm{Z}^n(X,\boldsymbol{\varphi}):={\bf
T}^{{\widehat{\otimes}}_{\bf R}\ n+1} {\widehat{\otimes}}_{\bf R}
X ,
$$
known as a (braided) cyclic ${\bf R}$-module tensor product.
A transposition morphism is an ${\bf R}$-bimodule map $w:{\bf T}\BP X \to {\bf
  T}\otimes_{\bf  R} X$, satisfying the conditions
\begin{eqnarray*}
&& w\circ (\boldsymbol{\varphi}\BP X)=\boldsymbol{\varphi}\otimes_{\bf R}
X\\
&&w\circ ({\overline{\bf t}}_i \BP X)= ({\overline{\bf t}}_o \otimes_{\bf R} X)
\circ ({\bf T} \otimes_{\bf R} w)\circ
{\overline{\mathfrak t}} X \circ({\bf T}\BP w),
\end{eqnarray*}
where the multiplication maps
$\overline{\bf t}_o:{\bf T}\otimes_{\bf R} {\bf T}\to {\bf T}$ and
$\overline{\bf t}_i:{\bf T}\otimes_{{\bf R}^\chi} {\bf T}\to {\bf T}$ are
induced by ${\bf t}:{\bf T}
\otimes {\bf T}\to {\bf T}$ and ${\bf t}\circ \chi_{{\bf T},{\bf T}}^{-1}:{\bf T}
\otimes {\bf T}\to {\bf T}$ , respectively, (cf. Example \ref{ex:bar_T}).
In degree $n$, the para-cocylic operator is given by
$$
w_n= \big({\bf T}^{{\widehat{\otimes}}_{\bf R}n}{\widehat{\otimes}}_{\bf R} w
\big) \circ \Pi {\overline{\mathfrak t}}^nX\circ {\mathfrak i}\big({{\bf
    T}^{\otimes_{\bf R}n} \otimes_{\bf R} X}\big),
$$
where ${\overline{\mathfrak t}}^nX$ is the projection of $\chi_{{\bf T},{\bf
    T}^{\otimes_{\bf R}\, n}}\otimes X$, i.e. the unique map ${\bf T} \BP ({\bf
    T}^{\otimes_{\bf R} \, n} \otimes_{\bf R} X) \to {\bf T}^{\otimes_{\bf R}
    \, n} \otimes_{\bf R} ({\bf T} \BP  X)$, for which
$$
P_{oi}\circ (\chi_{{\bf T},{\bf T}^{\otimes_{\bf R}\, n}}\otimes X)=
{\overline{\mathfrak t}}^nX \circ P_{io},
$$
where $P_{oi}: {\bf T}^{\otimes_{\bf R}\, n}\otimes {\bf T}\otimes X \to
{\bf T}^{\otimes_{\bf R}  \, n} \otimes_{\bf R} ({\bf T} \BP  X)$ and
$P_{io}:
{\bf T}\otimes {\bf T}^{\otimes_{\bf R}\, n}\otimes X \to
{\bf T} \BP ({\bf T}^{\otimes_{\bf R} \, n} \otimes_{\bf R} X)$
denote the canonical epimorphisms.

Examples of cosimplices in Hopf cyclic theory, corresponding to
(co)module algebras of a bialgebroid (or in particular a
bialgebra) over ${\bf R}$, are of this kind \cite{BohmStefan}. Further
examples are presented in Example \ref{cor:M^H} and Example \ref{cor:M_H} below.
\end{example}

\begin{example}\label{cor:M^H}
Let $H$ be a Hopf algebra with comultiplication $\Delta
:H\rightarrow H\otimes H$ and counit $\varepsilon :H\rightarrow
\mathbb{K}.$ For simplifying computations, we shall use the Sweedler-Heynemann 
$\Sigma$-notation $\Delta (h)=\sum h_{(1)}\otimes h_{(2)}.$

An important class of braided monoidal categories is provided by 
the categories ${\mathcal M}^H$ of right comodules over a coquasitriangular Hopf
algebra $H$. Cf. \cite[Chapter 10]{Mo}, a Hopf algebra $H$ is
\emph{coquasitriangular }if there is an invertible
(with respect to the convolution product) $\mathbb{K}$-linear map $
\left\langle -,-\right\rangle :H\otimes H\rightarrow \mathbb{K}$
such that,
for $h,k,l\in H$,
\begin{align}
\sum \langle h_{(1)},k_{(1)}\rangle k_{(2)}h_{(2)}& =\sum
h_{(1)}k_{(1)}\langle h_{(2)},k_{(2)}\rangle ,  \label{cqt1}\\
\langle h,kl\rangle & =\sum \langle h_{(1)},k\rangle \langle
h_{(2)},l\rangle ,  \label{cqt2}\\
\langle hk,l\rangle & =\sum \langle h,l_{(2)}\rangle \langle
k,l_{(1)}\rangle \label{cqt3}.
\end{align}
If, in addition, the inverse in convolution of $\left\langle
-,-\right\rangle $ is equal to $\left\langle -,-\right\rangle \circ \tau $,
we shall say that $H$ is \emph{cotriangular}. Here $\tau :H\otimes
H\rightarrow H\otimes H$ denotes the usual flip map. By
definition, $\langle -,-\rangle$ is called the coquasitriangular
map of $H$.

We fix a  coquasitriangular  Hopf algebra $(H,\langle
-,-\rangle)$ over a field ${\mathbb K}$. Let us briefly recall the
braided monoidal structure of ${\mathcal M}^H$,
that corresponds to $\langle -,-\rangle$.  First, for $(M,\rho
_{M})$ and $(N,\rho _{N})$ in $\mathcal{M}^{H},$
the tensor product of $M$ and $N$ in $\mathcal{M}^{H}$ is $M\otimes _{
\mathbb{K}}N,$ regarded as a comodule with respect to the right
diagonal coaction
\begin{equation*}
\rho (m\otimes n)=\sum m_{\left\langle 0\right\rangle }\otimes
n_{\left\langle 0\right\rangle }\otimes m_{\left\langle
1\right\rangle }n_{\left\langle 1\right\rangle }.
\end{equation*}
In this formula, for a right comodule $(M,\rho_M)$,  we used the
$\Sigma$-notation $\rho _{M}(m)=\sum m_{\left\langle
0\right\rangle }\otimes m_{\left\langle
1\right\rangle }$. The associativity and unity constraints in $\mathcal{M}
^{H}$ are induced by the corresponding structures in the monoidal
category of $\mathbb{K}$-linear spaces. The unit object is
$\mathbb{K},$ regarded as a trivial $H$-comodule. To define the
braiding in $\mathcal{M}^H$ we use the coquasitriangular map as
follows. For two right $H$-comodules $M$ and $N$ we define the
natural map
\begin{equation*}
\chi _{M,N}:M\otimes _{\mathbb{K}}N\rightarrow N\otimes
_{\mathbb{K}}M,\ \ \ \chi _{M,N}(m\otimes n)=\sum \left\langle
n_{\left\langle 1\right\rangle },m_{\left\langle 1\right\rangle
}\right\rangle n_{\left\langle 0\right\rangle }\otimes
m_{\left\langle 0\right\rangle }.
\end{equation*}
It is well-known that $\mathcal{M}^{H}$ is a braided monoidal
category with braiding $\chi$. Furthermore, $\mathcal{M}^{H}$ is a
symmetric monoidal category with respect to $\chi$ (that is $\chi _{N,M}\circ
\chi_{M,N} = \mathrm{Id} _{M\otimes N}$, for $N,M\in \mathcal{M}^{H}$), if and
only if $H$ is \emph{cotriangular}.

Commutative Hopf algebras are the simplest  examples of
cotriangular Hopf algebras. In this case the cotriangular
structure $\left\langle
-,-\right\rangle $ may be taken the trivial $\mathbb{K}$-linear map $
\left\langle k,k\right\rangle :=\varepsilon (h)\varepsilon (k),$ for any $
h,k\in H.$ Thus, the braiding is induced by $\tau ,$ the canonical
flip map.

Group Hopf algebras are not necessarily coquasitriangular. Indeed, let
$G$ denote a group and assume that $\left\langle -,-\right\rangle
$ is a coquasitriangular map on $\mathbb{K}G$.  Let $\gamma
:G\times G\rightarrow \mathbb{K}$ be the map
\begin{equation*}
\gamma (h,k):=\left\langle h,k\right\rangle ,\qquad \text{for
}h,k\in G.
\end{equation*}
Since the coquasitriangular map is invertible in convolution, it
follows
easily that $\left\langle h,k\right\rangle \in \mathbb{K}^{\ast },$ for any $
h,k\in G$. Hence $\gamma $ can be regarded as a map to
$\mathbb{K}^{\ast }.$ Clearly, then the condition (\ref{cqt1}) is
equivalent to the fact that $G$ is abelian. On the other hand,
relations (\ref{cqt2}) and (\ref{cqt3}) are equivalent to
\begin{equation*}
\gamma (hk,g)=\gamma (h,g)\gamma (k,g)\text{\ \ \ and~\ \ }\gamma
(g,hk)=\gamma (g,h)\gamma (g,k),\text{ }\ \ \forall h,g,k\in G.
\end{equation*}
A map $\gamma :G \times G \rightarrow K^{\ast },$
$\gamma (h,k)=\left\langle h,k\right\rangle $ satisfying the above
identities is called \emph{bi-character }of $G.$ In conclusion,
$\mathbb{K}G$ is coquasitriangular if, and only if, $G$ is abelian
and there is a bi-character $\gamma :G\times G\rightarrow
\mathbb{K}$ on $G$. As a matter of fact, we have also proved that,
for an abelian group $G$, there is an one-to-one correspondence
between the set of coquasitriangular structures on $\mathbb{K}G$
and the set of bi-characters on $G.$

The category $\mathcal{M}^{\mathbb{K}G}$ is equivalent to the category of $G$
-graded vector spaces. Therefore, an object in
$\mathcal{M}^{\mathbb{K}G}$ is a vector space $V$ together with a
decomposition $V:=\bigoplus_{g\in G}V_{g}$. The tensor product of
$V$ and $W$ in $\mathcal{M}^{\mathbb{K}G}$ is $V\otimes
_{\mathbb{K}}W\ $on which we take the decomposition
\begin{equation*}
(V\otimes _{\mathbb{K}}W)_{g}=\bigoplus\nolimits_{hk=g}V_{h}\otimes _{
\mathbb{K}}W_{k}.
\end{equation*}
The braiding $\chi _{V,W}:V\otimes _{\mathbb{K}}W\rightarrow W\otimes _{
\mathbb{K}}V$, for $v\in V_{h}$ and $w\in W_{k}$, is given by
\begin{equation*}
\chi _{V,W}(v\otimes w)=\gamma (k,h)w\otimes v.
\end{equation*}
Note that the bi-character $\gamma :G\times G\rightarrow
\mathbb{K}^{\ast }$ defines a cotriangular structure on
$\mathbb{K}G$ if, and only if, $\gamma $ is \emph{symmetric, }that
is, for $h\in G$ and $k\in G$,
\begin{equation*}
\gamma (h,k)^{-1}=\gamma (k,h).
\end{equation*}
Recall that the category of super vector spaces can be seen as the
symmetric monoidal category of  $\mathbb{Z}_2$-graded vector
spaces, whose braiding is induced by the bi-character
$\gamma:\mathbb{Z}_2\times\mathbb{Z}_2\to\mathbb{K}^\ast$,
$\gamma(g,h)=(-1)^{gh}$.
\bigskip

Our aim now is to specialize the construction in Example
\ref{ex:braid} to the case when $\boldsymbol{\varphi}:{\bf
R}\to{\bf T}$ is a braid preserving homomorphism of algebras in
$\mathcal{M} ^{H},$ for a coquasitriangular Hopf algebra $H$. In this case,
the conditions in \eqref{eq:br_ex} take the form
\begin{eqnarray}
&\boldsymbol{\varphi}(r) \otimes_{\mathbb K} t = 
\sum \langle {r_{\langle 1 \rangle}}_{(1)}, {t_{\langle 1 \rangle}}_{(1)} \rangle
 \langle {t_{\langle 1 \rangle}}_{(2)}, {r_{\langle 1 \rangle}}_{(2)} \rangle
\boldsymbol{\varphi}(r_{\langle 0 \rangle}) \otimes_{\mathbb K} t_{\langle 0 \rangle} 
\qquad &\textrm{and}\label{eq:2.1_ex}\\
&\boldsymbol{\varphi}(r) \otimes_{\mathbb K} r' = 
\sum \langle {r_{\langle 1 \rangle}}_{(1)}, {r'_{\langle 1 \rangle}}_{(1)} \rangle
 \langle {r'_{\langle 1 \rangle}}_{(2)}, {r_{\langle 1 \rangle}}_{(2)} \rangle
\boldsymbol{\varphi}(r_{\langle 0 \rangle}) \otimes_{\mathbb K} r'_{\langle 0
  \rangle}, 
\qquad &\textrm{for } r,r'\in {\bf R},\ t\in {\bf T}.
\nonumber
\end{eqnarray}
Clearly, $\overline{T}_{o}$ is the functor $\mathbf{T}\otimes _{\mathbf{R}
}(-).$ The multiplication in $\mathbf{R}^{\chi }$ is defined, for
$r^{\prime
}$ and $r^{\prime \prime }\mathbf{\ }$in $\mathbf{R}$, by
\begin{equation*}
r^{\prime }\cdot r^{\prime \prime }:=\sum \left\langle
r_{\left\langle 1\right\rangle }^{\prime \prime },r_{\left\langle
1\right\rangle }^{\prime }\right\rangle r_{\left\langle
0\right\rangle }^{\prime \prime }r_{\left\langle 0\right\rangle
}^{\prime }.
\end{equation*}
We have already noticed that $\mathbf{R}^{\chi }$ is an algebra in
$\mathcal{M}^H$ and that $\mathbf{T}$ is a right
$\mathbf{R}^{\chi }$-module with respect to the action
\begin{equation*}
t\cdot r:=\sum \left\langle r_{\left\langle 1\right\rangle
},t_{\left\langle 1\right\rangle }\right\rangle
\boldsymbol{\varphi}(r_{\left\langle 0\right\rangle
})t_{\left\langle 0\right\rangle }.
\end{equation*}
Furthermore, every $\mathbf{R}$-bimodule $X$ can be seen as a left
$\mathbf{R}^{\chi }$-module, with respect to the action
\begin{equation*}
r\cdot x:=\sum \left\langle x_{\left\langle 1\right\rangle
},r_{\left\langle 1\right\rangle }\right\rangle x_{\left\langle
0\right\rangle }r_{\left\langle 0\right\rangle },
\end{equation*}
where $x\in X$ and $r\in \mathbf{R}$. Hence $\overline{T}
_{i}:=\mathbf{T}\otimes _{\mathbf{R}^{\chi }}(-).$ For $X$ as
above, $\Pi X$
is the quotient of $X$ with respect to the vector space generated by the
commutators
\begin{equation*}
\lbrack x,r]:=rx-\sum \left\langle x_{\left\langle 1\right\rangle
},r_{\left\langle 1\right\rangle }\right\rangle x_{\left\langle
0\right\rangle }r_{\left\langle 0\right\rangle },
\end{equation*}
where $r$ and $x$ run arbitrarily in $\mathbf{R}$ and $X$,
respectively. For $t^{\prime },$ $t^{\prime \prime }\in \mathbf{T}$ and $
x\in X,$ the morphisms $\overline{\mathfrak{t}}X:\mathbf{T}\otimes _{\mathbf{
R}^{\chi }}(\mathbf{T}\otimes _{\mathbf{R}}X)\rightarrow \mathbf{T}\otimes _{
\mathbf{R}}(\mathbf{T}\otimes _{\mathbf{R}^{\chi }}X)$ and $\mathfrak{i}X:
\mathbf{T}\widehat{\otimes }_{\mathbf{R}}X\rightarrow $ $\Pi (\mathbf{T}
\otimes _{\mathbf{R}^{\chi }}X)$ are given by the following formulae
\begin{eqnarray*}
&&\overline{\mathfrak{t}}X(t^{\prime }\otimes _{\mathbf{R}^{\chi
}}t^{\prime \prime }\otimes _{\mathbf{R}}x)=\sum \left\langle
t_{\left\langle 1\right\rangle }^{\prime \prime },t_{\left\langle
1\right\rangle }^{\prime
}\right\rangle t_{\left\langle 0\right\rangle }^{\prime \prime }\otimes _{
\mathbf{R}}t_{\left\langle 0\right\rangle }^{\prime }\otimes _{\mathbf{R}
^{\chi }}x, \\
&&\mathfrak{i}X(t\widehat{\otimes }_{\mathbf{R}}x)=p(t\otimes _{\mathbf{R}
^{\chi }}x),
\end{eqnarray*}
where $p:\mathrm{Id}_{\mathbf{R}\text{-}\mathrm{Mod}\text{-}\mathbf{R}
}\rightarrow \Pi $ is the canonical projection and $\widehat{\otimes }_{
\mathbf{R}}$ denotes the braided cyclic tensor product introduced in Example 
\ref{ex:bimod_Pi}.

Let $X$ be an $\mathbf{R}$-bimodule in $\mathcal{M}^H$. For a transposition
morphism $w:\mathbf{T} \otimes _{\mathbf{R}^\chi} X\rightarrow \mathbf{T}\otimes
_{\mathbf{R}}X$, we use the notation
\begin{equation*}
w(t\otimes _{\mathbf{R}^{\chi }}x)=\sum t_{w}\otimes
_{\mathbf{R}}x_{w}.
\end{equation*}
Using this notation, the defining conditions of a transposition morphism are
equivalent to 
\begin{eqnarray*}
&&w(1\otimes _{\mathbf{R}^{\chi }}x)=1\otimes _{\mathbf{R}}x, 
\qquad \textrm{and}\qquad \\
&&\sum \left\langle t_{\left\langle 1\right\rangle }^{\prime \prime
},t_{\left\langle 1\right\rangle }^{\prime }\right\rangle \left(
t_{\left\langle 0\right\rangle }^{\prime \prime }t_{\left\langle
0\right\rangle }^{\prime }\right) _{w}\otimes
_{\mathbf{R}}x_{w}=\sum \left\langle \left( t_{w}^{\prime \prime
}\right) _{\left\langle 1\right\rangle },t_{\left\langle
1\right\rangle }^{\prime }\right\rangle \left( t_{w}^{\prime
\prime }\right) _{\left\langle 0\right\rangle }\left(
t_{\left\langle 0\right\rangle }^{\prime }\right) _{w^{\prime }}\otimes _{
\mathbf{R}}\left( x_{w}\right) _{w^{\prime }}.
\end{eqnarray*}
In degree $n$, the corresponding para-cocylic object is given by $\mathrm{Z}
^{n}(X,\boldsymbol{\varphi},w):=\mathbf{T}^{{\widehat{\otimes }}_{\mathbf{R}
}\ n+1}{\widehat{\otimes }}_{\mathbf{R}}X$. Its para-cocyclic
operator is
\begin{align*}
w_{n}(t^{0}{\widehat{\otimes }}_{\mathbf{R}}\cdots {\widehat{\otimes }}_{
\mathbf{R}}t^{n}{\widehat{\otimes }}_{\mathbf{R}}x)& =\sum
\left\langle t_{\left\langle 1\right\rangle }^{1},t_{\left\langle
n\right\rangle }^{0}\right\rangle \cdots \left\langle
t_{\left\langle 1\right\rangle }^{n},t_{\left\langle
1\right\rangle }^{0}\right\rangle t_{\left\langle
0\right\rangle }^{1}{\widehat{\otimes }}_{\mathbf{R}}\cdots {\widehat{
\otimes }}_{\mathbf{R}}t_{\left\langle 0\right\rangle
}^{n}{\widehat{\otimes
}}_{\mathbf{R}}\left( t_{\left\langle 0\right\rangle }^{0}\right) _{w}{
\widehat{\otimes }}_{\mathbf{R}}x_{w} \\
& =\sum \left\langle t_{\left\langle 1\right\rangle }^{1}\cdots
t_{\left\langle 1\right\rangle }^{n},t_{\left\langle
1\right\rangle
}^{0}\right\rangle t_{\left\langle 0\right\rangle }^{1}{\widehat{\otimes }}_{
\mathbf{R}}\cdots {\widehat{\otimes }}_{\mathbf{R}}t_{\left\langle
0\right\rangle }^{n}{\widehat{\otimes }}_{\mathbf{R}}\left(
t_{\left\langle 0\right\rangle }^{0}\right) _{w}{\widehat{\otimes
}}_{\mathbf{R}}x_{w}.
\end{align*}
In the particular case when $H:=\mathbb{K}G$ and the braiding is
defined by a bi-character $\gamma ,$ conditions \eqref{eq:2.1_ex} reduce to
$$
\boldsymbol{\varphi}(r) \otimes_{\mathbb K} t = \gamma(k,h) \gamma(h,k)
\boldsymbol{\varphi}(r) \otimes_{\mathbb K}  t  \qquad \textrm{and}\qquad
\boldsymbol{\varphi}(r) \otimes_{\mathbb K} r'= \gamma(k,h) \gamma(h,k)
\boldsymbol{\varphi}(r) \otimes_{\mathbb K} r',
$$
for $h,k \in G$, $r\in {\bf R}_h$, $r' \in {\bf R}_k$ and $t\in {\bf
T}_k$. Note that, if some component ${\bf R}_h$ lies within the kernel of
$\boldsymbol{\varphi}$, then these conditions may hold also for a non-trivial
braiding between copies of ${\bf R}$ and ${\bf T}$.
The second condition in the definition of a transposition map becomes
\begin{equation*}
\sum \gamma \left( k,h\right) \left( t^{\prime \prime
}t^\prime\right) _{w}\otimes _{\mathbf{R}}x_{w}=\gamma \left( g,h\right) \left(
t_{w}^{\prime \prime }\right) _{u}t_{w^{\prime }}^{\prime }\otimes _{\mathbf{
R}}\left( x_{w}\right) _{w^{\prime }},
\end{equation*}
where $h,k,g\in G$ and $t^{\prime }\in \mathbf{T}_{h},$ $t^{\prime \prime }\in
\mathbf{T}_{k} $ and 
$x$ belongs to the component $X_g$ of the $G$-graded vector
space $X$.
In this particular case, for $t^{i}\in \mathbf{T}_{g_{i}},$ the
para-cocyclic operator satisfies
\begin{equation*}
w_{n}(t^{0}{\widehat{\otimes }}_{\mathbf{R}}\cdots {\widehat{\otimes }}_{
\mathbf{R}}t^{n}{\widehat{\otimes }}_{\mathbf{R}}x)=\sum \gamma
\left(
g_{1}\cdots g_{n},g_{0}\right) t^{1}{\widehat{\otimes }}_{\mathbf{R}}\cdots {
\widehat{\otimes }}_{\mathbf{R}}t^{n}{\widehat{\otimes }}_{\mathbf{R}
}t_{w}^{0}{\widehat{\otimes }}_{\mathbf{R}}x_{w}.
\end{equation*}
\end{example}

\begin{example}
\label{cor:M_H}Dually, the category ${}_H {\mathcal M}$ of left modules over a
${\mathbb K}$-Hopf algebra $H$ is braided monoidal if, and only if, $H$ is
quasitriangular. Cf. \cite[Chapter 10]{Mo}, a Hopf algebra $H$ is
\emph{quasitriangular }if there is an
invertible element $R:=\sum_{i}a_{i}\otimes b_{i}$ in $H\otimes H$ such that
\begin{align*}
\sum h_{(2)}\otimes h_{(1)}& =R(\sum h_{(1)}\otimes h_{(2)})R^{-1}
 \\
(\Delta \otimes \mathrm{Id}_{H})(R)& =R^{13}R^{23},  \\
(\mathrm{Id}_{H}\otimes \Delta )(R)& =R^{13}R^{12},
\end{align*}
where $R^{12}=\sum_i a_{i}\otimes b_{i}\otimes 1$ and $R^{13}$ and
$R^{23}$ are defined analogously. If $R^{-1}:=\sum_i b_{i}\otimes
a_{i}$, then we say that $H$ is \emph{triangular}.

The braided monoidal structure that corresponds to a
quasitriangular element $R$ is defined as follows. For $M$
and $N$ in $_{H}\mathcal{M},$ the tensor product of $M$ and $N$ in $_{H}
\mathcal{M}$ is $M\otimes _{\mathbb{K}}N,$ regarded as a module
with respect
to the left diagonal action
\begin{equation*}
h\triangleright (m\otimes n)=\sum h_{(1)}\triangleright m\otimes
h_{(2)}\triangleright n.
\end{equation*}
The associativity and unity constraints in $_{H}\mathcal{M}$ are
induced by the corresponding structures in the monoidal category
of $\mathbb{K}
$-linear spaces. The unit object is $\mathbb{K},$ regarded as a trivial $H$
-module. The functorial morphism of left $H$-modules
\begin{equation*}
\chi _{M,N}:M\otimes _{\mathbb{K}}N\rightarrow N\otimes
_{\mathbb{K}}M,\ \ \ \chi _{M,N}(m\otimes
n)=\sum_{j}\left(c_{j}\triangleright n\right)\otimes
\left(d_{j}\triangleright m\right)
\end{equation*}
is a braiding on $_{H}\mathcal{M,}$ where $R^{-1}=\sum_{j}c_{j}\otimes d_{j}$.

Let us specialize the para-cocyclic object constructed in Example \ref
{ex:braid} to a homomorphism of algebras $\boldsymbol{\varphi}:\mathbf{R}
\rightarrow \mathbf{T}$ in the braided monoidal category ${}_{H} \mathcal{M}$.
In this case, conditions \eqref{eq:br_ex} read as
$$
\boldsymbol{\varphi}(r) \otimes_{\mathbb K} t =
\sum_{j,k} (c_k d_j \triangleright \boldsymbol{\varphi}(r)) \otimes_{\mathbb
  K} (d_k c_j \triangleright t) \qquad \textrm{and}\qquad 
\boldsymbol{\varphi}(r) \otimes_{\mathbb K} r' =
\sum_{j,k} (c_k d_j \triangleright \boldsymbol{\varphi}(r)) \otimes_{\mathbb
  K} (d_k c_j \triangleright r'),
$$
for $r,r'\in {\bf R}$ and $t\in {\bf T}$, where $\sum_{j}c_{j}\otimes d_{j} =
R^{-1} = \sum_{k} c_{k}\otimes d_{k}$.
Clearly, $\overline{T}_{o}$ is the functor $\mathbf{T}\otimes _{\mathbf{R}
}(-).$ Moreover, the multiplication in $\mathbf{R}^{\chi }$ is defined, for $
r^{\prime }$ and $r^{\prime \prime }\mathbf{\ }$in $\mathbf{R}$, by
\begin{equation*}
r^{\prime }\cdot r^{\prime \prime }:=\sum_{j}(c_{j}\triangleright
r^{\prime \prime })(d_{j}\triangleright r^{\prime }).
\end{equation*}
We have already noticed that $\mathbf{R}^{\chi }$ is an algebra in ${}_H
\mathcal{M}$ and $\mathbf{T}$ is a right $\mathbf{R}^{\chi }$-module with
respect to the action
\begin{equation*}
t\cdot r:=\sum \left( c_{j}\triangleright
\boldsymbol{\varphi}(r)\right) \left( d_{j}\triangleright t\right)
.
\end{equation*}
Furthermore, if $X$ is an $\mathbf{R}$-bimodule then, for $x\in X$
and $r\in \mathbf{R,}$ the following formula
\begin{equation*}
r\cdot x:=\sum_{j}\left( c_{j}\triangleright x\right) \left(
d_{j}\triangleright r\right)
\end{equation*}
defines a left $\mathbf{R}^{\chi }$-module structure. Hence $\overline{T}
_{i}:=\mathbf{T}\otimes _{\mathbf{R}^{\chi }}(-).$ For $X$ as
above, $\Pi X$
is the quotient of $X$ with respect to the vector space generated by the
commutators 
\begin{equation*}
\lbrack x,r]:=rx-\sum_{j}\left( c_{j}\triangleright x\right)
\left( d_{j}\triangleright r\right) ,
\end{equation*}
where $r$ and $x$ run arbitrarily in $\mathbf{R}$ and $X$,
respectively. For $t^{\prime },$ $t^{\prime \prime }\in \mathbf{T}$ and $
x\in X,$ the morphisms $\overline{\mathfrak{t}}X:\mathbf{T}\otimes _{\mathbf{
R}^{\chi }}(\mathbf{T}\otimes _{\mathbf{R}}X)\rightarrow \mathbf{T}\otimes _{
\mathbf{R}}(\mathbf{T}\otimes _{\mathbf{R}^{\chi }}X)$ and $\mathfrak{i}X:
\mathbf{T}\widehat{\otimes }_{\mathbf{R}}X\rightarrow $ $\Pi (\mathbf{T}
\otimes _{\mathbf{R}^{\chi }}X)$ are given by the following formulae
\begin{eqnarray*}
&&\overline{\mathfrak{t}}X(t^{\prime }\otimes _{\mathbf{R}^{\chi
}}t^{\prime \prime }\otimes
_{\mathbf{R}}x)=\sum_{j}c_{j}\triangleright t^{\prime \prime
}\otimes _{\mathbf{R}}d_{j}\triangleright t^{\prime }\otimes _{\mathbf{R}
^{\chi }}x, \\
&&\mathfrak{i}X(t\widehat{\otimes }_{\mathbf{R}}x)=p(t\otimes _{\mathbf{R}
^{\chi }}x),
\end{eqnarray*}
where $p:\mathrm{Id}_{\mathbf{R}\text{-}\mathrm{Mod}\text{-}\mathbf{R}
}\rightarrow \Pi $ is the canonical projection and $\widehat{\otimes }_{
\mathbf{R}}$ denotes the braided cyclic tensor product introduced in Example 
\ref{ex:bimod_Pi}.

Let $X$ be an $\mathbf{R}$-bimodule in ${}_H \mathcal{M}$. For a transposition
morphism $w:\mathbf{T} \otimes _{\mathbf{R}^{\chi }}X\rightarrow
\mathbf{T}\otimes_{\mathbf{R}}X$, we use the notation
\begin{equation*}
w(t\otimes _{\mathbf{R}^{\chi }}x)=\sum t_{w}\otimes
_{\mathbf{R}}x_{w}.
\end{equation*}
Using this notation, the defining conditions of a transposition map are
equivalent to 
\begin{eqnarray*}
&&w(1\otimes _{\mathbf{R}^{\chi }}x)=1\otimes _{\mathbf{R}}x, 
\qquad \textrm{and}\qquad \\
&&\sum_{j}\left[ \left( c_{j}\triangleright t^{\prime \prime
}\right) \left(
d_{j}\triangleright t\right) \right] _{w}\otimes _{\mathbf{R}
}x_{w}=\sum_{j}\left( c_{j}\triangleright t_{w}^{\prime \prime
}\right)
\left( d_{j}\triangleright t^{\prime }\right) _{w^{\prime }}\otimes _{
\mathbf{R}}\left( x_{w}\right) _{w^{\prime }}.
\end{eqnarray*}
In degree $n$, the corresponding para-cocylic object is given by $\mathrm{Z}
^{n}(X,\boldsymbol{\varphi},w):=\mathbf{T}^{{\widehat{\otimes }}_{\mathbf{R}
}\ n+1}{\widehat{\otimes }}_{\mathbf{R}}X$ and its para-cocyclic operator is
\begin{align*}
w_{n}(t_{0}{\widehat{\otimes }}_{\mathbf{R}}\cdots {\widehat{\otimes }}_{
\mathbf{R}}t_{n}{\widehat{\otimes
}}_{\mathbf{R}}x)=\sum_{j_{1},\dots
j_{n}}\left( c_{j_{1}}^{(1)}\triangleright t_{1}\right) &{\widehat{\otimes }}
_{\mathbf{R}}\cdots {\widehat{\otimes }}_{\mathbf{R}}\left(
c_{j_{n}}^{(n)}\triangleright t_{n}\right) {\widehat{\otimes
}}_{\mathbf{R}}
\\
&{\widehat{\otimes }}_{\mathbf{R}}\left[
d_{j_{n}}^{(n)}\triangleright \left( \cdots \triangleright \left(
d_{j_{1}}^{(1)}\triangleright t_{0}\right) \cdots \right) \right]
_{w}{\widehat{\otimes }}_{\mathbf{R}}x_{w},
\end{align*}
where $\sum_{j_{k}}c_{j_{k}}^{(k)}\otimes d_{j_{k}}^{(k)}=R^{-1}$,
for every $k=1,\dots ,n.$
\end{example}

\end{document}